\documentclass[journal]{IEEEtran}

\usepackage{cite}

\ifCLASSINFOpdf
  \usepackage[pdftex]{graphicx}
  \graphicspath{figspdf/}
  \DeclareGraphicsExtensions{.pdf,.jpeg,.png}
\else
  \usepackage[dvips]{graphicx}
  \graphicspath{figspdf/}
  \DeclareGraphicsExtensions{.eps}
\fi

\usepackage{amsmath,amsthm,amssymb,mathtools,enumitem,url}
\usepackage{booktabs,tabularx,makecell,colortbl,multirow}
\newcolumntype{C}{>{\centering\arraybackslash}X}
\newcolumntype{R}{>{\raggedleft\arraybackslash}X}
\newcolumntype{L}{>{\raggedright\arraybackslash}X}

\DeclareMathOperator*{\argmin}{arg\,min}

\newcommand{\im}{\mathrm{i}}

\newtheorem{proposition}{Proposition}

\usepackage{xcolor}
\usepackage{soul} %
\usepackage[hidelinks,draft=false]{hyperref} %
\usepackage{float} %
\usepackage[caption=false]{subfig} %
\usepackage{dashbox} %
\usepackage{stfloats,multicol}

\begin{document}
\title{Failure Probability Constrained\\AC Optimal Power Flow}
\author{Anirudh~Subramanyam,
        Jacob~Roth,
        Albert~Lam,
        and~Mihai~Anitescu,~\IEEEmembership{Member,~IEEE}%
\thanks{The authors are with the Mathematics and Computer Science Division, Argonne National Laboratory, Lemont IL 60439.}}%

\markboth{}%
{Subramanyam, Roth, Lam and~Anitescu: %
    Failure Probability Constrained ACOPF}

\maketitle

\begin{abstract}
    Despite cascading failures being the central cause of blackouts in power transmission systems, existing operational and planning decisions are made largely by ignoring their underlying cascade potential. %
    This paper posits a reliability-aware AC Optimal Power Flow formulation that seeks to design a dispatch point which has a low operator-specified likelihood of triggering a cascade starting from any single component outage.
    By exploiting a recently developed analytical model of the probability of component failure,
    our Failure Probability-constrained ACOPF (FP-ACOPF) utilizes the system's expected first failure time as a smoothly tunable and interpretable signature of cascade risk.
    We use techniques from bilevel optimization and numerical linear algebra to efficiently formulate and solve the FP-ACOPF using off-the-shelf solvers.
    Extensive simulations on the IEEE 118-bus case show that, when compared to the unconstrained and N-1 security-constrained ACOPF, our probability-constrained dispatch points %
    can significantly lower the probabilities of long severe cascades and of large demand losses, while incurring only minor increases in total generation costs.
\end{abstract}


\IEEEpeerreviewmaketitle

\section{Introduction}

\IEEEPARstart{A}{} cascading failure in a power transmission system refers to a sequence of dependent outages of individual system components that successively disable parts of the grid, leading to a significant loss of served power or large blackout in the worst case.
Accounts of major blackouts reveal that cascading failures are often triggered by an initial event that is largely unpredictable (such as extreme weather) but are sustained by subsequent events that are causally linked via Kirchhoff's laws and automatic control actions of protection devices.
For example, the outage of a single component can lead to redistribution of power flows in the remainder of the network in a way that can cause large overcurrents on some transmission lines.
This, in turn, may trigger protection relays to disconnect these lines automatically if the current flow exceeds some threshold rating, or it may lead to eventual thermal failure if the overcurrents remain sustained for a long time.

The large direct and indirect costs associated with blackouts, along with mandatory standards set forth by the North American Electric Reliability Corporation (NERC) to address cascading outages, %
have motivated the development of a plethora of tools for the simulation and analysis of cascading failures; see reviews~\cite{baldick2008initial,papic2011survey,guo2017critical,haes2019survey} %
and references therein.
These tools can be broadly classified as
\emph{(a)} complex network approaches that consider the pure topological properties of power networks while ignoring or simplifying the underlying physics (e.g., see~\cite{hines2008centrality,brummitt2012suppressing}), %
\emph{(b)} high-level statistical models built either on historical/simulation data or on simplified power system physics (e.g., see~\cite{dobson2005loading,rahnamay2014stochastic}), %
\emph{(c)} quasi-steady-state methodologies that utilize DC or AC power flow (steady state) models, typically in conjunction with models of protection mechanisms or operator interventions \cite{yu2004practical,chen2005cascading,mei2008study,anghel2007stochastic,yang2020dynamic}, and rely on Monte Carlo or enumerative sampling (e.g., see~\cite{nedic2006criticality,yan2014cascading,henneaux2018benchmarking}), %
\emph{(d)} dynamics-based models in which the state evolution is resolved with a physics-based representation of the grid dynamics (e.g., see~\cite{demarco2001phase,song2015dynamic}) to study post-fault transient dynamics \cite{schafer2018dynamically}, synchronization \cite{dorfler2010synchronization,nishikawa2015comparative,rohden2016cascading}, or behavior over longer time horizons \cite{fabozzi2009simplified,khaitan2009fast,abhyankar2012real}. %

While these tools have proved immensely useful in deepening our understanding of cascading failures, their use for risk mitigation and decision-making have been largely restricted to limiting the propagation of a cascading failure, e.g., via controlled load shedding or intentional islanding, \emph{after} severe contingencies have already occurred (e.g., see~\cite{bienstock2011optimal,esmaeilian2016prevention,guo2020localization}). %
Existing practices for the prevention of cascading failures \emph{before} they occur, by tuning and modifying the controllable properties of the power network, has largely relied upon $N-k$ security criteria and simulation-based contingency analyses as notional surrogates for reducing cascading likelihood. %

Although simulation-based tools can influence long-term cascade mitigation solutions, such as line capacity or generator margin allocations and protection system enhancements (e.g., see~\cite{wang2001optimal,hardiman2004advanced,qi2014interaction}), %
they are fundamentally limited in preventing cascades in short-term operations such as economic dispatch or optimal power flow.
This is because they do not provide any direct functional relationship between control parameters and cascade potential, or they entail expensive Monte Carlo sampling and numerical integration requirments and are thus challenging to incorporate within optimization algorithms. %

This work proposes to incorporate in the classical ACOPF model, an analytical -- as opposed to simulation-based -- model of cascade severity that is an \emph{explicit function} of the network properties and dispatch point. %
In contrast to existing methods, we aim to determine a dispatch point subject to the constraint that the \emph{probability of individual component failure} remains below an operator-prescribed threshold.
Our model capitalizes upon results from~\cite{roth2019kinetic}, which in contrast to other approaches for simulating cascading failures, provides an \emph{analytic expression} for the failure probability as a function of the dispatch point.
This is achieved by modeling \emph{Gaussian load and generation fluctuations} in the AC power flow dynamics, and interpreting the latter as the diffusion of a particle in an energy landscape subject to stochastic forcing.
Large deviations theory then provides the means to analytically relate the failure probability to the underlying energy surface.

For a given system state (consisting of voltages and power flows) at equilibrium, the analytical expression for the failure probability of an individual component requires solving a nonlinear optimization problem that computes a ``most likely'' failure state starting from this equilibrium state.
However, since the latter is contingent on the dispatch point, explicitly constraining the failure probability of an \emph{individual component}, or of a \emph{cascading sequence of multiple components}, is tantamount to solving a bilevel optimization problem.
We settle for the former
and demonstrate that constraining individual failure probabilities, which is equivalent to increasing the system's expected first failure time--or decreasing its \textit{failure rate}--%
can be an effective surrogate for constraining the probability of a cascading failure sequence. %
Solving the bilevel model, however, is computationally challenging since it is nonlinear, nonconvex, and  %
involves eigenvalues and
determinants of high-dimensional matrices that scale with the network size.
Nevertheless, we show that this eigenvalue- and determinant-constrained model can be reformulated entirely algebraically and solved efficiently using standard solvers, if we exploit the low-rank nature of the failure constraints %
along with the first- and second-order optimality conditions of the nested problem.

Our key finding is that it is possible to modify the dispatch point so as to satisfy a prescribed threshold of failure probability. %
Specifically, the system's expected first failure time, and hence, the probability of cascading over a given time horizon, can be lowered %
significantly without incurring vastly higher generation costs or load shedding.

We believe this is the first work that %
captures some notion of cascading risk in operational dispatch.
Although it can be viewed as a probabilistic $N-1$ variant, our approach offers several advantages over classical $N-k$ models.
The first crucial difference is that, by tuning the system's \emph{failure rate limit}, we can systematically and \emph{smoothly} explore the trade-offs between cascade potential, dispatch costs and operator conservatism, {without significant increase in computational complexity}.
In contrast, $N-k$ approaches must resort to a non-smooth control of $k$ to achieve the same objective, while invariably incurring a sharp increase in combinatorial complexity.
Another subtle, yet practically useful, advantage of our approach is its \emph{interpretability}.
Indeed, the benefit (in terms of reliability) of increasing $k$ in $N-k$ approaches, is difficult to convey outside the domain.
Our approach, on the other hand, allows the system operator to decide between a \textit{failure rate limit} of $10^{-6} s^{-1}$ or $10^{-15} s^{-1}$ (for example), which is equivalent to deciding between observing the first failure once every $10^6$ or $10^{15}$ seconds.
This is a statement that is better aligned with the philosophy of regulatory constraints which tend to be in occurrences per unit of time.

The paper is organized as follows.
Section~\ref{sec:probability_model} presents assumptions, and reviews the failure probability model of~\cite{roth2019kinetic}. %
Section~\ref{sec:optimization_model} presents the failure probability-constrained ACOPF model along with its reformulation.
Section~\ref{sec:numerical_simulations} demonstrates the effectiveness of our method via extensive simulations, and Section~\ref{sec:conclusions} offers conclusions and directions for future work.

\section{Failure Probability Model}\label{sec:probability_model}

\subsection{Notation}\label{sec:probability_model:notation}
We use
$\mathcal{N}~=~\{1, \ldots, n_b\}$ to denote the set of buses, %
and $\mathcal{L} \subseteq \mathcal{N} \times \mathcal{N}$ %
to denote the set of transmission lines, where $l = (i,j) \in \mathcal{L}$ is a line from bus $i$ and to bus $j$.
We denote the set of generators by $\mathcal{G}~=~\{1, \ldots, n_g\}$, %
and let $\mathcal{G}_i$ be the set of generators connected to bus $i \in \mathcal{N}$; note that $\mathcal{G} = \cup_{i \in \mathcal{N}} \mathcal{G}_i$.
For ease of notation, we define $\mathcal{N}^{\prime} \coloneqq \{i \in \mathcal{N}: \mathcal{G}_i = \emptyset\}$ to denote the set of non-generator buses, \emph{i.e.}, those that are not connected to any generator and similarly,
$\mathcal{L}^{\prime} \coloneqq \{ (i, j) \in \mathcal{L} : \mathcal{G}_i = \emptyset \text{ or } \mathcal{G}_j = \emptyset\}$ to denote the set of lines that are connected to at least one non-generator bus.
The nodal admittance matrix $Y = G + \sqrt{-1} B \in \mathbb{C}^{n_b \times n_b}$ with conductance and susceptance matrices $G \in \mathbb{R}^{n_b \times n_b}$ and $B \in \mathbb{R}^{n_b \times n_b}$, respectively. %

We denote active and reactive power generations by $p_g, q_g \in \mathbb{R}^{n_g}$,
active and reactive power demands by $p_d, q_d \in \mathbb{R}^{n_b}$,
and net active and reactive powers by $p_{net}, q_{net} \in \mathbb{R}^{n_b}$, where we define $p_{net,i} \coloneqq p_{d,i} - \sum_{k \in \mathcal{G}_i} p_{g,k}$ and $q_{net,i} \coloneqq q_{d,i} - \sum_{k \in \mathcal{G}_i} q_{g,k}$ for $i \in \mathcal{N}$. %
The voltage magnitudes and angles are denoted by $V \in \mathbb{R}^{n_b}$ and $\theta \in \mathbb{R}^{n_b}$,
and the generator angular velocities are denoted by $\omega \in \mathbb{R}^{n_g}$.

For a vector $z \in \mathbb{C}^n$, we use $\lVert z \rVert$ and $z^*$ to denote its Euclidean norm and Hermitian transpose, respectively, and we use $e^{z}$ and $\log(z)$ to denote the vectors $(e^{z_1}, \ldots, e^{z_n}) \in \mathbb{C}^n$ and $(\log(z_1), \ldots, \log(z_n)) \in \mathbb{C}^n$.
Given vectors $z, \tilde{z} \in \mathbb{C}^n$, we use $z \circ \tilde{z}$ to denote the Hadamard product $(z_1 \tilde{z}_i, \ldots, z_n \tilde{z}_n) \in \mathbb{C}^n$.
For a matrix $A \in \mathbb{C}^{n \times n}$, we use $\rho(A)$, $\det(A)$, and $\mathop{\mathrm{adj}}(A)$ to denote its spectral radius, determinant and adjugate, respectively, and $A \succ 0$ ($\succeq 0$) to indicate that it is positive definite (semi-definite).
For a given matrix $A \succeq 0$ and vector $z \in \mathbb{C}^n$, we use $\lVert z \rVert_A$ to denote $\sqrt{z^* A z}$.

In formulating the probability model, it will be convenient to divide the %
    system state vector $(p_g, q_g, V, \theta, \omega) \in \mathbb{R}^m$ ($m \coloneqq 3n_g + 2n_b$)
into two distinct sets.
After selecting an arbitrary generator bus $\sigma \in \mathcal{N} \setminus \mathcal{N}^{\prime}$, $|\mathcal{G}_{\sigma}| = 1$, as the slack bus in the steady-state context and reference bus in the dynamics context,
we collect all voltage magnitudes at non-generator buses, and phase angles and angular velocities at non-slack buses in the state vector
$x = \left(
\{V_i\}_{i \in \mathcal{N}^{\prime}}, \{\theta_i\}_{i \in \mathcal{N} \setminus \{\sigma\}}, \{\omega_i\}_{i \in \mathcal{G} \setminus \mathcal{G}_{\sigma}}
\right)
\in \mathbb{R}^{d}$, $d = |\mathcal{N}^{\prime}| + |\mathcal{N}| + |\mathcal{G}| - 2$.
The remaining voltage components, angular velocities and power generations are aggregated in the vector
$y = \left(
\{V_i\}_{i \in \mathcal{N} \setminus \mathcal{N}^{\prime}}, \theta_\sigma, \omega_\sigma, p_g, q_g
\right)
\in \mathbb{R}^{m - d}
$.
This partition of the system variables, illustrated in Table~\ref{table:x_y_notation}, is purely for notational convenience.
In particular, any specified value of $y$ %
implicitly defines a set of state vectors $x$ that solve the power flow equations, as we shall elaborate in Section~\ref{sec:probability_model:main_results}. %

\begin{table}[!t]
    \caption{Partition of system variables into sub-vectors $x$ and $y$.}
    \label{table:x_y_notation}
    \centering
\begin{tabularx}{\columnwidth}{cCCCCCcCCCc}
    \toprule
    & \multicolumn{2}{c}{\makecell{Load buses\\$\mathcal{N}^{\prime}$}} & \multicolumn{4}{c}{\makecell{Generator buses\\$\mathcal{N} \setminus (\mathcal{N}^{\prime} \cup \{\sigma\})$}} & \multicolumn{4}{c}{\makecell{Slack bus\\$\sigma$}} \\
    \cmidrule(lr){2-3} \cmidrule(lr){4-7} \cmidrule(l){8-11} 
        & $\theta$     & $V$          & $\omega$     & $\theta$     & $V$          & $p_g, q_g$   & $\omega$     & $\theta$     & $V$      & $p_g, q_g$ \\ \midrule 
    $x$ & $\checkmark$ & $\checkmark$ & $\checkmark$ & $\checkmark$ &              &              &              &              &          & \\
    $y$ &              &              &              &              & $\checkmark$ & $\checkmark$ & $\checkmark$ & $\checkmark$ & $\checkmark$ & $\checkmark$ \\
    \bottomrule
\end{tabularx}
\end{table}

\subsection{Assumptions}\label{sec:probability_model:assumptions}
Motivated by the automatic control actions of protection relays, we assume that individual system components fail in a deterministic manner according to a component-specific state equation.
Specifically, we assume that a component, such as a transmission line $l \in \mathcal{L}$, fails if the value of a certain function $\Theta_l : \mathbb{R}^{m} \mapsto \mathbb{R}$ at the current state $(x, y)$ exceeds a critical value $\Theta_l^{\max}$.
Furthermore, we assume that after failure, a component remains failed over the entire dispatch horizon of the ACOPF.

The algebraic state functions $\Theta$ can be any continuously differentiable function of the state vector $(x, y)$, and hence, it can be used to model various component failures, such as under-voltage load shedding at buses or line surges in transformers.
For the purpose of computing a dispatch point, however, we assume that once initiated, component failures occur only due to relay trips caused by current overloads in transmission lines.
Suppose that $I_{l}^{\mathrm{trip}}$ denotes the emergency current rating of a line $l = (i, j) \in \mathcal{L}$.
In this case, one can define $\Theta_l(x, y)$ to be the square of the magnitude of current flow in line $l$ (which depends only on the voltages and phase angles at its terminal buses), and $\Theta_l^{\max}$ to be the square of the corresponding rating:
\begin{equation}\label{eq:line_energy}
    \begin{array}{rl}
    \displaystyle\Theta_l(x, y)
    &= \displaystyle\left(|Y_{ij}| \left\lvert V_i e^{\im \theta_i} - V_j e^{\im \theta_j} \right\rvert\right)^2 \\
    &= \displaystyle|Y_{ij}|^2 \left( V_i^2 + V_j^2 - 2 V_i V_j \cos (\theta_i - \theta_j) \right),  \\
    \displaystyle\Theta_l^{\max} &= (I_{l}^{\mathrm{trip}})^2.
    \end{array}
\end{equation}%
This is partly because it simplifies the exposition, and partly because the corresponding failure model has already been studied and validated against real cascade data in~\cite{roth2019kinetic,zheng_2010_bistable}.
We note, however, that the methodology does not preclude us from considering %
more general definitions of $\Theta$.
Importantly, the simulations in Section~\ref{sec:numerical_simulations} model several other protection mechanisms beyond current overloads.

In addition, we shall assume that
\textit{(i)} the network is lossless\footnote{%
    This assumption is not particularly restrictive since many high-voltage power transmission networks typically have resistance/reactance ratios below 0.2.
    Moreover, the standard DCOPF model for dispatch widely used in industry explicitly relies on this assumption~\cite{stott2009dc}; %
    in contrast, our model is far more general since we also consider nonlinearities, reactive powers and voltage magnitudes that the DCOPF model does not.
}, \textit{i.e.}, $G = 0$,
and
\textit{(ii)} only the subset $\mathcal{L}^{\prime}$ of transmission lines that connect to at least one non-generator bus have a nonzero likelihood of failing %
because of stochastic fluctuations in active and reactive power demand.

\subsection{Analytical model of line failures}\label{sec:probability_model:main_results}
We model the grid's electro-mechanical behavior using a system of stochastic differential equations (SDE), by introducing a scalar noise parameter $\tau > 0$ into the following variant%
of the standard, structure-preserving model \cite{demarco_1987_security,bergen_1981_structure,demarco_singular_1984,demarco_1985_small,pai_1989_energy}:
\begin{equation}\label{eq:sde}
\mathrm{d}x^{\tau}_t = (J-S)\nabla_x \mathcal{H}(x^{\tau}_t,y) \mathrm{d}t + \sqrt{2\tau S} \mathrm{d}W_t.
\end{equation}
Here, $x_t^{\tau} \in \mathbb{R}^d$ denotes the system state at time $t$,
$J$ and $S$ represent appropriate system matrices, and %
$W_t$ is a $d$-dimensional vector of independent Wiener processes. %
    When $\tau = 0$, equation \eqref{eq:sde} models the deterministic dynamics of the %
        state variables in $x$ as a function of the input vector $y$.
    Specifically, it is a singularly-perturbed version \cite{demarco_singular_1984} of the classical differential-algebraic structure-preserving model \cite[Chapter~7]{sauer1997power}. %
    When $\tau = 0$, equation \eqref{eq:sde} includes the standard second-order swing equations employed in the classical model, but augments its constant active power load modeling assumption by adding linear frequency-dependent load damping terms $D_l$ to the active power loads (see, e.g., \cite{arif2017load}) and voltage-dependent perturbation terms $D_v$ to the reactive power loads.
    The parameters $D_l$ and $D_v$ control the rates at which the phase angles and load voltages approach the real and reactive power flow equation manifolds, respectively \cite{bergen_1981_structure}.
    Crucially, equation \eqref{eq:sde} converges to the classical model as the noise parameter $\tau \to 0$ and damping terms $D_l, D_v \to 0$ but with the added advantages of global well-posedness (when $\tau = 0$)
    and modeling Gaussian perturbations in active and reactive power demands (when $\tau > 0$).
    This model has been widely used for transient stability analysis in both deterministic and stochastic regimes \cite{demarco_1987_security,demarco1990energy}.

    The detailed differential equations encapsulated within \eqref{eq:sde}, including definitions of $J$ and $S$, can be found in \cite{demarco_1987_security,zheng_2016_new,roth2019kinetic}. 
    The energy function $\mathcal{H}$ is then obtained as the first integral of the deterministic dynamics, leading to the gradient-based formulation \eqref{eq:sde}.
    Notably, the integral admits the following closed-form expression %
    :
\begin{equation}\label{eq:energy_function}
    \begin{aligned}
\mathcal{H}(x, y) \coloneqq & \tfrac{1}{2}\omega^{\top} M \omega + \tfrac{1}{2} (V \circ e^{\im \theta})^{*}\, Y\, (V \circ e^{\im \theta}) \\
&+ p_{net}^{\top} \theta + q_{net}^{\top} \log(V),
\end{aligned}
\end{equation}
where $M$ is the $n_g \times n_g$ diagonal matrix of generator masses.
Unlike equation \eqref{eq:line_energy}, the right-hand side of equation \eqref{eq:energy_function} depends on all components of $x$ and $y$.
It is worth pointing out that \eqref{eq:energy_function} has been widely used as a Lyapunov function for transient stability analysis and dynamic control of the structure preserving swing dynamics model in deterministic settings \cite{DemarcoLyapunov1988,sun2016direct},
and it is closely related to the widely-used transient energy function of \cite{fouad1988transient}. Similar to the latter, which is used to measure distance to instability, we use \eqref{eq:energy_function} to measure the probability of the stochastically perturbed dynamics to reach a triggering surface.

We note that the vector $y$ is assumed to have been fixed to appropriate values \textit{a priori}, and we shall return to choosing a value for $y$ in the subsequent section when we describe the ACOPF formulation.
For now, observe that the structure of the energy function \eqref{eq:energy_function} ensures that, for any specified value of $y$, any point $x \in \mathbb{R}^d$ that satisfies $\nabla_x \mathcal{H}(x, y) = 0$ is a solution to the lossless power flow equations:
\begin{gather}
p_{net,i} + \sum_{j \in \mathcal{N}} B_{ij} V_i V_j \sin(\theta_i - \theta_j) = 0, \;\; i \in \mathcal{N}, \label{eq:real_power_balance} \\
q_{net,i} - \sum_{j \in \mathcal{N}} B_{ij} V_i V_j \cos(\theta_i - \theta_j) = 0, \;\; i \in \mathcal{N}. \label{eq:imag_power_balance}
\end{gather}
In particular, any local energy minimizer,
\begin{equation}\label{eq:xbar}
\bar{x}(y) \in \argmin_{x \in \mathbb{R}^d} \mathcal{H}(x, y)
\end{equation}
defines a feasible solution to~\eqref{eq:real_power_balance}, \eqref{eq:imag_power_balance} and serves as a stable equilibrium %
for the network dynamics defined by $y$.
To simplify notation, we drop the dependence of $\bar{x}$ on $y$ in the remainder of the paper.

Under the SDE model~\eqref{eq:sde} initialized at $x_0^{\tau} = \bar{x}$, the failure probability of a line $l \in \mathcal{L}^{\prime}$ (considered in isolation from the rest of the network) can be quantified by its so-called \emph{first passage time}.
The latter is defined to be the first time at which the system state $x_t^{\tau}$ triggers the failure condition \eqref{eq:line_energy} for line~$l$.
Assuming $\Theta_l(\bar{x}, y) < \Theta_l^{\max}$, this is equivalent to
\begin{equation}\label{eq:first_passage_time}
T_{l}^{\tau}(y) \coloneqq \inf \{ t>0 : \Theta_l(x_t^{\tau}, y) \geq \Theta_l^{\max} \}.
\end{equation}
By exploiting results from large deviations theory~\cite{wf_2012_random,collet_2013_qst}, it was shown in \cite{roth2019kinetic} that, as $\tau \to 0$, $T_{l}^{\tau}(y)$ is an exponential random variable whose mean satisfies the following relation:
\begin{equation}\label{eq:fw-limit}
\lim_{\tau \to 0} \tau \log \left( \mathbb{E} \left[ T_{l}^{\tau} (y) \right] \right) =\hspace{-0.5em} \min_{x: \Theta_l(x_t^{\tau}, y) = \Theta_l^{\max}} \hspace{-0.5em}\mathcal{H}(x, y) - \mathcal{H}(\bar{x}, y)
\end{equation}
A point at which the minimum is obtained is called the \emph{most likely failure point} and is defined as
\begin{equation}\label{eq:xstar}
x^{\star}_l(y) \coloneqq \argmin_{x \in \mathbb{R}^d} \left\{\mathcal{H}(x, y) : \Theta_l(x, y) = \Theta_l^{\max} \right\},
\end{equation}
where we have implicitly assumed $x^{\star}_l$ to be unique\footnote{This is not a strong assumption based on the validation in~\cite{roth2019kinetic}.}.
As before, we drop the dependence of $x^{\star}_l$ on $y$ to simplify notation.

For finite values of $\tau$, we can thus use relation \eqref{eq:fw-limit} to compute a log-approximate \emph{failure rate} (\textit{i.e.}, the reciprocal of the mean failure time).
In addition to this log-approximation, a subexponential correction to the failure rate was also obtained in \cite{roth2019kinetic}
yielding the following expressions:
\begin{align}\label{eq:failure_rate_first}
\lambda_l^{{\tau}} (y) &\coloneqq \mathsf{pf}_{l}^1(y) \times \mathsf{ef}_l (y)\\
\mathsf{pf}^1_l(y) &\coloneqq \mathsf{pf}_l^0(y) \times \left(1 + \frac{\tau} { \mathcal{H}(x^{\star}_l,y) - \mathcal{H}(\bar{x},y)} \right). \label{eq:prefactor_first} \\
\mathsf{pf}^0_l (y) &\coloneqq %
\left\lVert\nabla_x \mathcal{H}(x^{\star}_l, y)\right\rVert_S^2 \sqrt{ \frac{ \det \nabla^2_{xx} \mathcal{H}(\bar{x}, y) }{2 \pi \tau C^{\star}_l (y)}} \label{eq:prefactor_zeroth}\\
\mathsf{ef}_l (y) &\coloneqq \exp \left [ - \frac{\mathcal{H}(x^{\star}_l, y) - \mathcal{H}(\bar{x}, y)}{\tau} \right ], \label{eq:energy_factor}
\end{align}
where $C^{\star}_l (y)$ is a factor accounting for the curvature of the failure surface in the vicinity of $x^{\star}_l$, and is given by:
\begin{align}
C^{\star}_l(y) &\coloneqq \nabla^\top_x \mathcal{H}(x^{\star}_l, y)
\mathop{\mathrm{adj}}\left(X_l(y)\right)
\nabla_x \mathcal{H}(x^{\star}_l, y) \label{eq:B-star} \\
X_l(y) &\coloneqq \nabla^2_{xx} \mathcal{H}(x^{\star}_l, y) - \mu^{\star}_l \nabla^2_{xx} \Theta_l (x^{\star}_l, y), \label{eq:L-star}
\end{align}
where
$\mu^{\star}_l \in \mathbb{R}$ is the optimal Lagrange multiplier in the constrained optimization problem~\eqref{eq:xstar}. %
In equations \eqref{eq:failure_rate_first}--\eqref{eq:energy_factor}, $\mathsf{pf}$ and $\mathsf{ef}$ stand for the \emph{pre-factor} and \emph{exponential factor}, respectively.
    The latter follows directly from equation \eqref{eq:fw-limit}.
    The former is a subexponential correction that is obtained by quantifying how the probability distribution over the system's state space changes with time.
    Specifically, from an initial \emph{quasistationary ensemble} \cite{collet_2013_qst} of particles distributed as $p_{qst}^{\tau}(x) \propto \exp \left[ -\tau^{-1} \left(\mathcal{H}(x, y) - \mathcal{H}(\bar{x}, y)\right) \right]$, probability mass is lost over time through the surface $\{x \in \mathbb{R}^d: \Theta_l(x, y) = \Theta_l^{\max}\}$.
    A Laplace integral approximation of this probability yields the (zeroth order) failure rate:
    $
    \mathsf{pf}_{l}^0(y) \times \mathsf{ef}_l (y).
    $
    However, for large values of $\tau$, this zeroth order expression can lead to non-physical behavior as the $1/\sqrt{\tau}$ term in $\mathsf{pf}_{l}^0(y)$ may cause the failure rate to decrease as $\tau$ increases.
    Therefore, the first order correction in equation \eqref{eq:prefactor_first} is used to remedy this behavior, and it has demonstrated strong empirical performance in approximating the failure rate across a wide range of $\tau$ \cite{roth2019kinetic}.%

In summary, the distribution of failure times for line $l \in \mathcal{L}'$ can be well-approximated by
$
    T_{l}^{\tau} (y) \sim \mathrm{Exp}(\lambda_l^{\tau} (y)),
$
and therefore, the probability of observing line $l$ fail in the first $t$ seconds can be given by
\begin{align}\label{eq:failure_probability}
    \mathbb{P}(T_{l}^{\tau} (y) \leq t) = 1 - \exp \left[ -\lambda_l^{\tau} (y) t \right].
\end{align}
A noteworthy feature of this model is that \eqref{eq:failure_probability} gives a closed-form expression for the failure probability, thus avoiding the need for direct integration of the differential model \eqref{eq:sde}.

\section{Failure Probability-constrained AC Optimal Power Flow Formulation}\label{sec:optimization_model}

\subsection{Conceptual formulation}\label{sec:optimization_model:conceptual}
The probability of transmission line failure~\eqref{eq:failure_probability} is a function of $y = \left(
\{V_i\}_{i \in \mathcal{N} \setminus \mathcal{N}^{\prime}}, \theta_\sigma, p_g, q_g
\right)
$,
the voltage magnitudes at generator buses, the voltage angle at the slack bus and the active and reactive power generations.
The choice of $y$ also determines the equilibrium operating point $\bar{x}(y)$
through the power flow equations~\eqref{eq:real_power_balance}, \eqref{eq:imag_power_balance};
we shall denote these simply as
$
x = 
\left(
\{V_i\}_{i \in \mathcal{N}^{\prime}}, \{\theta_i\}_{i \in \mathcal{N} \setminus \{\sigma\}},
\right)
$
and along with the vector $y$, they constitute decision variables\footnote{Note that we do not include the frequency variables $\omega$ as decisions, since we can set $\omega = 0$ without loss of optimality.} in our Failure Probability-constrained ACOPF (FP-ACOPF) formulation:
\begin{alignat}{1}
\mathop{\text{minimize}}_{x, y} \;& \sum_{k \in \mathcal{G}} c_k(p_{g,k}) \label{eq:generation_cost} \\
\text{subject to} \;& \eqref{eq:real_power_balance}-\eqref{eq:imag_power_balance}, \notag \\
& p_{g,k}^\mathrm{min} \leq p_{g,k} \leq p_{g,k}^\mathrm{max}, \;\;k \in \mathcal{G}, \label{eq:real_power_gen_bounds} \\
& q_{g,k}^\mathrm{min} \leq q_{g,k} \leq q_{g,k}^\mathrm{max}, \;\;k \in \mathcal{G}, \label{eq:imag_power_gen_bounds} \\
& V_{i}^\mathrm{min} \leq V_{i} \leq V_{i}^\mathrm{max}, \;\; i \in \mathcal{N}, \label{eq:voltage_magnitude_bounds} \\
& \Theta_l(x, y) \leq (I_{l}^{\mathrm{lim}})^2, \;\; l \in \mathcal{L}, \label{eq:line_flow_bounds} \\
& \lambda_l^{\tau} (y) \leq \lambda_l^\mathrm{lim} \coloneqq -t_H^{-1} \log\left( 1-\epsilon_{l}^{\mathrm{lim}} \right), l\in \mathcal{L}' \label{eq:rate_constraint_conceptual}
\end{alignat}
The objective function~\eqref{eq:generation_cost} minimizes the cost of generation where $c_k$ is a function representing the cost of generating active power $p_{g,k}$ at generator $k \in \mathcal{G}$.
The constraints~\eqref{eq:real_power_gen_bounds}, \eqref{eq:imag_power_gen_bounds} and \eqref{eq:voltage_magnitude_bounds} ensure that the power generations and voltage magnitudes stay within predefined limits.
The constraint~\eqref{eq:line_flow_bounds} limits the amount of current flow on the transmission lines, where we have distinguished the flow limit $I_{l}^{\mathrm{lim}}$ from the value $I_{l}^{\mathrm{trip}}$ used in the failure model~\eqref{eq:line_energy}.
The latter quantity is the value at which protection relays would automatically disconnect the line, and it is typically much larger than the former, which is determined based on thermal considerations.
Finally, constraint~\eqref{eq:rate_constraint_conceptual} imposes an upper bound on the failure rate of line $l \in \mathcal{L}'$, which by definition, is equivalent to delaying its expected failure time to be greater than $\left(\lambda_l^\mathrm{lim}\right)^{-1}$.

Observe that, using expression~\eqref{eq:failure_probability},
the failure rate constraint~\eqref{eq:rate_constraint_conceptual} is equivalent to a chance constraint
which ensures that the probability
of line $l \in \mathcal{L}'$ failing over a time horizon $t_H$
is less than some operator-prescribed limit $\epsilon_{l}^{\mathrm{lim}}$ (say 1\%):
\begin{equation}
\mathbb{P}(T_{l}^{\tau} (y) \leq t_H) \leq \epsilon_{l}^{\mathrm{lim}}, \;\; l\in \mathcal{L}'.
\label{eq:chance_constraint_conceptual}
\end{equation}
Constraint~\eqref{eq:chance_constraint_conceptual} only limits the failure probability of a single line.
While it can also be used as an approximation for limiting the probability of a cascading failure sequence involving more than one line\footnote{%
    For example, limiting the probability that line~$l_2$ will fail after line~$l$ within time $t_H$ can be enforced via:
        \begin{align*}
            \int_{0}^{t_H} \mathbb{P}\left( T_{l}^{\tau} (y) \in [t, t + \mathrm{d}t] \right) \mathbb{P}\left( T_{l_2}^{\tau} (y; Y') \in [t, t_H] \right) \mathrm{d}t \leq \epsilon^{\mathrm{lim}},
        \end{align*}
        where $T_{l_2}^{\tau} (y; Y')$ is the failure time of line~$l_2$ calculated using a modified admittance matrix $Y'$ after removing line~$l$.
        Since both probabilities in the above integral involve exponential random variables \eqref{eq:failure_probability}, it can be analytically expressed as a single algebraic constraint using the corresponding failure rates.
        Observe that this probability is upper bounded by the left-hand side of~\eqref{eq:chance_constraint_conceptual}.%
},
we highlight that assessing the true probability of a cascading sequence needs careful consideration and modeling of subsequent events, including system dynamics and protection behavior, which is outside the scope of the proposed optimization model.
The individual failure rate constraint~\eqref{eq:rate_constraint_conceptual} is thus closer in spirit to a probabilistic $N-1$ constraint, as opposed to a full-fledged cascading failure constraint.

The presented model is a standard ACOPF formulation with the exception of the \emph{failure rate constraint}~\eqref{eq:rate_constraint_conceptual}. %
We now present an efficient reformulation of these rate constraints that can be incorporated into standard optimization solvers.

\subsection{Reformulation of failure rate constraints: Key steps}\label{sec:optimization_model:reformulation}

Constraining the failure rate via~\eqref{eq:rate_constraint_conceptual} is nontrivial because it  involves \textit{(i)} the solution of a nested nonconvex optimization problem~\eqref{eq:xstar} to obtain $x^\star_l(y)$, and \textit{(ii)} constraints involving determinants of large matrices, namely $\nabla^2_{xx} \mathcal{H}(x, y)$ and $X_l(y)$, see~\eqref{eq:prefactor_zeroth} and~\eqref{eq:B-star}.
This section shows that these can be circumvented by exploiting the low-rank property of the line failure function $\Theta_l$, and a partial Talyor expansion of the energy function $\mathcal{H}$ around the equilibrium operating point~$x$.

\subsubsection{Low rank factorization of failure function}
We shall capitalize on the fact that, the Hessian of the failure function $\Theta_l$ admits a low-rank factorization of the form:
\begin{equation}\label{eq:low_rank_factorization}
\nabla_{xx}^2 \Theta_l (x, y) = Q_l(x, y) \, K_l (x, y) \, Q_l^\top (x, y),
\end{equation}
where $Q_l (x, y)$ is a $d \times r_l$ matrix, with $r_l \ll m$ and $K_l(x, y)$ is a $r_l \times r_l$ diagonal matrix.
Observe that this low-rank structure is natural whenever $\Theta_l$ is a function of a small number of state variables, as is the case when $\Theta_l$ models the failure of transmission lines, system buses, generators or transformers.

Although such a low-rank factorization can be computed numerically (\emph{e.g.}, using an eigenvalue decomposition), it can also be obtained analytically in several cases, such as in the examples of Section~\ref{sec:probability_model:assumptions}.
To illustrate this, Appendix~\ref{appendix:closed_form_low_rank} presents explicit closed-form expressions for $Q_l$ and $K_l$ when $\Theta_l$ is the line overcurrent function~\eqref{eq:line_energy}.
In particular, it shows that $r_l = 3$ for a line $l = (i, j)$ where both $i, j \in \mathcal{N}'$ are non-generator buses, whereas $r_l = 2$ otherwise.
The fact that the rank $r_l \leq 3$ has nothing to do with the specific functional form of~\eqref{eq:line_energy}; in fact, $r_l \leq 3$ will always hold whenever $\Theta_l$ is a function only of $V_i, V_j$ and the angle difference $\theta_i - \theta_j$.

\subsubsection{Taylor approximation of energy function}
The value of the energy function at the most likely failure point, $\mathcal{H}(x^\star_l, y)$, is approximated using a quadratic Taylor polynomial centered around the equilibrium operating point~$x$ that is determined by the ACOPF formulation:
\begin{align}
\hspace{-0.5em}
\mathcal{H}(x^\star_l, y)
&= \mathcal{H}(x, y) + \frac12 (x^\star_l - x)^\top \nabla^2_{xx} \mathcal{H}(x, y) (x^\star_l - x) \label{eq:energy_function_quadratic}
\end{align}
where we have ignored the first-order term since
the FP-ACOPF constrains~$x$ to satisfy the power flow equations~\eqref{eq:real_power_balance}, \eqref{eq:imag_power_balance}, which are equivalent to $\nabla_x \mathcal{H} (x, y) = 0$.
We shall comment on the accuracy of~\eqref{eq:energy_function_quadratic} in the next subsection.

\subsection{Reformulation of the `most likely failure point' problem}
Equations~\eqref{eq:low_rank_factorization} and~\eqref{eq:energy_function_quadratic} can be used to reformulate the nested nonconvex \emph{most likely failure point} optimization problem~\eqref{eq:xstar}.
In particular, the following proposition shows that they can be used to generate tractable algebraic constraints under which the solution of~\eqref{eq:xstar} can be characterized by its first- and second-order optimality conditions
\begin{proposition}\label{prop:bilevel_reformulation}
    Fix a line $l \in \mathcal{L}'$, and assume that $x, x^\star_l \in \mathbb{R}^d$ and $y \in \mathbb{R}^{m - d}$ are such that:
    \begin{enumerate}
        \item $\nabla^2_{xx} \mathcal{H} (x, y) \succ 0$, and
        \item the Taylor approximation~\eqref{eq:energy_function_quadratic} is applicable.
    \end{enumerate}
    If $x^\star_l$ and $\mu^{\star}_l \geq 0$ satisfy~\eqref{eq:bilevel_kkt_stationarity}--\eqref{eq:bilevel_sosc}, then $x^\star_l$ is a local solution of the optimization problem~\eqref{eq:xstar}.
    \begin{align}
    & \nabla^2_{xx} \mathcal{H}(x, y) (x^\star_l - x) = \mu^{\star}_l \nabla_x \Theta_l (x^\star_l, y) \label{eq:bilevel_kkt_stationarity}\\
    & \Theta_l (x^\star_l, y) = \Theta_l^{\max} \label{eq:bilevel_feasibility} \\
    & \mu^{\star}_l \, \rho \left( 
    A_l
    \right) < 1, \label{eq:bilevel_sosc}
    \end{align}
    where $A_l \coloneqq K_l(x^\star_l, y)\, Q_l^\top (x^\star_l, y)\, [\nabla^{2}_{xx} \mathcal{H}(x, y) ]^{-1} \, Q_l (x^\star_l, y)$ is a $r_l \times r_l$ matrix.
\end{proposition}

\begin{proof}
    See Appendix~\ref{appendix:proof_bilevel}.
\end{proof}

Equations~\eqref{eq:bilevel_kkt_stationarity} and~\eqref{eq:bilevel_feasibility} are the first-order optimality conditions, whereas the spectral inequality~\eqref{eq:bilevel_sosc} is derived from the second-order optimality condition of the nested problem~\eqref{eq:xstar}.
It is noteworthy that the matrix $A_l$ in the left-hand side of the latter inequality is a small $r_l \times r_l$ matrix, and hence it is possible to reformulate the inequality {in a completely algebraic form without requiring a numerical eigenvalue computation.}

When $r_l = 2$, it can be verified (\emph{e.g.}, using the characteristic polynomial of $A_l$) that inequality~\eqref{eq:bilevel_sosc} is equivalent to:
\begin{equation}\label{eq:spectral_inequality_r2}
\mu^{\star}_l \mathop{\mathrm{tr}}(A_l) - (\mu^{\star}_l)^2 \det(A_l) < 1,
\end{equation}
where the trace and determinant can be computed in algebraic closed-form, since $A_l$ is a $2 \times 2$ matrix.

When $r_l = 3$, it is nontrivial to get a tractable reformulation, since $A_l$ need not be symmetric.
We consider two cases.
\begin{itemize}
    \item If $K_l(x^\star_l, y) \succ 0$, then $\mu^{\star}_l \rho(A_l) < 1$ is equivalent to $\mu^{\star}_l \rho(\hat{A}_l) < 1$, $\hat{A}_l \coloneqq K_l^\frac12\, Q_l^\top \, [\nabla^{2}_{xx} \mathcal{H}(x, y) ]^{-1} \, Q_l K_l^\frac12$. This is a symmetric matrix, and hence, we can use Sylvester's criterion to enforce that all $r_l$ leading principal minors of $I - \mu^{\star}_l \hat{A}_l$ (which are computable in algebraic closed-form) must be positive.
    A similar trick can be used if $K_l(x^\star_l, y)$ is negative definite.
    
    \item If $K_l(x^\star_l, y)$ is indefinite, then it is difficult to enforce the spectral inequality, and we settle for a relaxation which enforces a constraint on the sum of its $r_l$ eigenvalues:
    \begin{equation}\label{eq:spectral_inequality_r3}
    \mu^{\star}_l \mathop{\mathrm{tr}}(A_l) < r_l,
    \end{equation}
    where again the trace can be computed in closed-form.
\end{itemize}

We close with some remarks on the assumptions of Proposition~\ref{prop:bilevel_reformulation}.
The first assumption requires that the Hessian of the energy function at the equilibrium point~$x$ is positive definite.
Note that it is satisfied whenever~$x$
is obtained via the energy minimization problem~\eqref{eq:xbar}.
Although using the power flow equations \eqref{eq:real_power_balance}, \eqref{eq:imag_power_balance} as a surrogate for the latter cannot guarantee this condition, we found that it was almost never violated in our experiments.
With regards to the second assumption, the Taylor approximation~\eqref{eq:energy_function_quadratic} is guaranteed to be accurate only when $x^\star_l$ is near $x$.
We provide empirical evidence in Section~\ref{sec:numerical_simulations} that the approximation is not severe.

\subsection{Treatment of high-dimensional determinants}

Using equations \eqref{eq:failure_rate_first} and \eqref{eq:prefactor_first}, the failure rate constraint \eqref{eq:rate_constraint_conceptual} can be equivalently written as:
\begin{align*}
\lambda_l^{\tau} (y)
&=
\mathsf{pf}_l^0 (y) \cdot \mathsf{ef}_l(y) \cdot \left(1 + \frac{\tau} { \mathcal{H}(x^{\star}_l,y) - \mathcal{H}(\bar{x},y)} \right)
\leq 
\lambda_l^\mathrm{lim}.
\end{align*}
The formula for the prefactor $\mathsf{pf}_l^0 (y)$ that appears in the above constraint involves determinants of (the typically large) $d\times d$ matrices, $\nabla^2_{xx} \mathcal{H}(x, y)$ and $X_l(y)$, via \eqref{eq:prefactor_zeroth} and \eqref{eq:B-star} respectively.
The presence of these determinants can slow down computation of the rate constraints and their gradients, and may result in ill-conditioning.
Fortunately, equations~\eqref{eq:low_rank_factorization} and~\eqref{eq:energy_function_quadratic} can be used to circumvent these issues.
The following proposition uses the matrix-determinant lemma~\cite{Harville2008} to obtain a reformulation of the prefactor that avoids large determinants.

\begin{proposition}\label{prop:determinants}
    Fix a line $l \in \mathcal{L}'$, and assume that $x, x^\star_l \in \mathbb{R}^d$ and $y \in \mathbb{R}^{m - d}$ satisfy the conditions of Proposition~\ref{prop:bilevel_reformulation}.
    If $x^\star_l$ and $\mu^{\star}_l \geq 0$ satisfy~\eqref{eq:bilevel_kkt_stationarity}--\eqref{eq:bilevel_sosc}, then the prefactor and energy factor in equation~\eqref{eq:failure_rate_first} admit the following reformulation:
    \begin{align}
    \mathsf{pf}_l^1 (y) &= \mathsf{pf}_l^0 (y) \left(1 + \frac{2\tau}{\mu^{\star}_l \alpha_l}\right) \label{eq:prefactor_first_reformulated} \\
    \mathsf{pf}_l^0 (y) &= \frac{(\mu^{\star}_l)^{3/2} \lVert \nabla_x \Theta_l (x^\star_l, y) \rVert_S^2}{\sqrt{2\pi \tau \left( \alpha_l \det (W_l) + \beta_l\right)}} \label{eq:prefactor_zeroth_reformulated} \\
    \mathsf{ef}_l(y) &= \exp\left(-\frac{\mu^{\star}_l\alpha_l}{2\tau}\right), \label{eq:energy_factor_reformulated}
    \end{align}
    where $W_l \coloneqq I - \mu^{\star}_l A_l$ is a $r_l \times r_l$ matrix, $A_l$ is as in Proposition~\ref{prop:bilevel_reformulation}, and the scalars $\alpha_l \coloneqq (x^\star_l - x)^\top \nabla_x \Theta_l (x^\star_l, y)$ and $\beta_l \coloneqq (x^\star_l - x)^\top\, Q_l(x^\star_l, y) \, \mathop{\mathrm{adj}}(W_l)\, K_l(x^\star_l, y)\, Q_l^\top (x^\star_l, y) \, (x^\star_l - x)$.
\end{proposition}

\begin{proof}
    See Appendix~\ref{appendix:proof_determinant}.
\end{proof}

Unlike equations \eqref{eq:prefactor_zeroth} and \eqref{eq:B-star}, equation \eqref{eq:prefactor_zeroth_reformulated} involves the determinant and adjugate of the small $r_l \times r_l$ matrix $W_l$.
This allows efficient computation of the scalar $\beta_l$, and hence, of the overall failure rate and its gradient.
Moreover, it avoids the ill-conditioning that plagues the original constraint, and allows expressing it in an algebraic optimization system.

\subsection{Efficient practical implementation}
We now highlight other key algorithmic enhancements that are necessary to improve the practical performance of the FP-ACOPF. %
First, since the absolute value of the failure rates $\lambda_l^{\tau} (y)$ can be very small, we found it beneficial to implement the rate constraint~\eqref{eq:rate_constraint_conceptual} in its log-form.
Using equations \eqref{eq:failure_rate_first}, and \eqref{eq:prefactor_first_reformulated}--\eqref{eq:energy_factor_reformulated}, this is equivalent to:
\begin{align}
&\frac32 \log (\mu^{\star}_l) + \log \left( \lVert \nabla_x \Theta_l (x^\star_l, y) \rVert_S^2 \right) - \frac12 \log \left( \alpha_l \det (W_l) + \beta_l\right) \notag \\
&+\log \left(1 + \frac{2\tau}{\mu^{\star}_l\alpha_l} \right) - \frac{\mu^{\star}_l \alpha_l}{2\tau} 
\leq \log(\sqrt{2\pi \tau}\lambda_l^\mathrm{lim}). \label{eq:rate_constraint_log}
\end{align}
All intermediate terms involved in the left-hand side of~\eqref{eq:rate_constraint_log} are computable in closed form, except for the matrix $A_l$ which appears in the definition of $W_l \coloneqq I - \mu^{\star}_l A_l$.
A close examination of the formula for $A_l$ in Proposition~\ref{prop:bilevel_reformulation} reveals that this matrix can be efficiently constructed by first solving for the intermediate matrix $Z_l \in \mathbb{R}^{d \times r_l}$ as follows:
\begin{equation}\label{eq:Zl_linear_solve}
\nabla^{2}_{xx} \mathcal{H}(x, y) \cdot Z_l = Q_l (x^\star_l, y),
\end{equation}
and then setting $A_l = K_l(x^\star_l, y)\, Q_l^\top (x^\star_l, y) \, Z_l$.
We therefore introduce explicit decision variables $Z_l$, along with \eqref{eq:Zl_linear_solve} as constraints in the FP-ACOPF formulation.
Notably, this allows the complete rate constraint~\eqref{eq:rate_constraint_log}, including all of its intermediate expressions, to be computed algebraically.

Implementing the rate constraint~\eqref{eq:rate_constraint_conceptual} for a particular line $l \in \mathcal{L}'$ thus involves introducing $d r_l + d + 1$ additional variables, namely $Z_l \in \mathbb{R}^{d \times r_l}$, $x^\star_l \in \mathbb{R}^d$ and $\mu^{\star}_l \geq 0$ and $d r_l + d + 3$ additional constraints, namely \eqref{eq:bilevel_kkt_stationarity}--\eqref{eq:bilevel_sosc} and \eqref{eq:rate_constraint_log}--\eqref{eq:Zl_linear_solve}.
In practice, the rate constraints~\eqref{eq:rate_constraint_log} are binding for only a small fraction of critical lines.
Therefore, adding $\mathcal{O}(dr_l)$ constraints for the remaining non-critical lines unnecessarily increases model complexity.
To limit this growth, we add additional variables and constraints in an iterative manner, only for those lines $l \in \hat{\mathcal{L}}$ whose rate constraints~\eqref{eq:rate_constraint_log} are found to be violated.
Our algorithm can be described as follows.

\begin{enumerate}[leftmargin=4em,labelindent=16pt,label=\bfseries Step \arabic*:]
    \item Set $\hat{\mathcal{L}} = \emptyset$. %
    \item Define decision variables
    $x \in \mathbb{R}^d$, 
    $y \in \mathbb{R}^{m-d}$, and
    $\big\{(x^\star_l, \mu^{\star}_l, Z_l) \in \mathbb{R}^d \times \mathbb{R}_{\geq 0} \times \mathbb{R}^{d \times r_l} \big\}$ for each $l \in \hat{\mathcal{L}}$.\\
    Compute an optimal solution $\hat{z}$ of the problem:
    \begin{align*}\label{eq:acopf_iterative}
    \mathop{\text{minimize}} &\;\; \eqref{eq:generation_cost} \\
    \text{subject to} &\;\; \eqref{eq:real_power_balance}, \eqref{eq:imag_power_balance},
    \eqref{eq:real_power_gen_bounds}-\eqref{eq:line_flow_bounds}, \\
    &\;\; \eqref{eq:bilevel_kkt_stationarity}-\eqref{eq:bilevel_sosc},
    \eqref{eq:rate_constraint_log}, \eqref{eq:Zl_linear_solve}  \;\; \forall  l \in \hat{\mathcal{L}}.
    \end{align*}
    
    \item For all $l \in \mathcal{L}'\setminus \hat{\mathcal{L}}$ (possibly in parallel):
    \begin{enumerate}
        \item Use $\hat{z}$ to compute candidate values of $x^\star_l, \mu^{\star}_l, Z_l$ by solving equations \eqref{eq:bilevel_kkt_stationarity}-\eqref{eq:bilevel_sosc}, and
        \eqref{eq:Zl_linear_solve}.
        \item If inequality \eqref{eq:rate_constraint_log} is violated using $\hat{z}$ and the computed values of $x^\star_l, \mu^{\star}_l, Z_l$, update $\hat{\mathcal{L}} = \hat{\mathcal{L}} \cup \{l\}$ and save the computed values as warm-starts.
    \end{enumerate}
    
    \item If $\hat{\mathcal{L}}$ was updated, go to Step~2; %
    else, stop and output $(\hat{x}, \hat{y})$ as the optimal FP-ACOPF solution.
\end{enumerate}
\section{Numerical Simulations}\label{sec:numerical_simulations}

We now present results and insights obtained from several experiments.\footnote{Our Julia code is available at: \href{https://github.com/jacob-roth/OPF}{github.com/jacob-roth/OPF}.
All optimization problems were solved using Ipopt %
with linear solver MA27
via JuMP.
The runs were performed on 32 threads of an Intel Xeon Gold 6140 CPU at 2.30GHz with 512 GB RAM.
} %
The IEEE 118-bus test system with line limits $I^\mathrm{lim}_l$ obtained from  PGLib \cite{pglib}, is used in all experiments
except Section~\ref{sec:numerical_simulations:computational_efficiency} where we also include larger networks.
We set branch and shunt conductances to zero to satisfy the lossless assumption\footnote{For some context on this assumption, the mean branch %
resistance/reactance ratio is approximately 0.2.
Also, the mean %
difference between the ACOPF solutions computed with and without resistances is 0.004 p.u. %
and %
0.04 for the voltage magnitudes and angles, respectively.
},
remove transmission line taps for simplicity,
and set the system matrices $J$ and $S$ that appear in \eqref{eq:sde} as per \cite{zheng_2016_new}\footnote{We set generator masses $M = 0.0531$, generator damping $D_g = 0.05$, load damping $D_l = 0.005$ and perturbation parameter $D_v = 0.01$.}.
The numerical performance of the FP-ACOPF %
is largely unaffected by the choice of the noise parameter $\tau$, the line failure threshold $I^\mathrm{trip}_l$, and the network loading $(p_d, q_d)$ levels.
Therefore, Sections \ref{sec:numerical_simulations:failure_rate_analysis}, \ref{sec:numerical_simulations:computational_efficiency}, \ref{sec:numerical_simulations:kmc_simulations} and \ref{sec:numerical_simulations:accfm_simulations} use `baseline' values of $\tau = 10^{-4}$, $I^\mathrm{trip}_l = 1.10\, I^\mathrm{lim}_l$ and $(p_d, q_d)$ values from MATPOWER \cite{matpower}.
Their values can, however, strongly affect cascading behavior; therefore, Section~\ref{sec:numerical_simulations:sensitivity} includes a discussion of cascading sensitivity to their values.

\subsection{Analysis of line failure rates}\label{sec:numerical_simulations:failure_rate_analysis}
For a fixed dispatch point $y$, the true line failure rates $\lambda^\tau_l(y)$ are computed by solving problem~\eqref{eq:xstar} to obtain the most likely failure point $x^\star_l$ and its multiplier $\mu^{\star}_l$.
However, the FP-ACOPF computes an estimate, which we denote as $\tilde{\lambda}^\tau_l (y)$, by solving~\eqref{eq:bilevel_kkt_stationarity}--\eqref{eq:bilevel_sosc} instead of \eqref{eq:xstar}.
Fig.~\ref{fig:rates:approximation_quality} shows the corresponding approximation error for the subset of lines with the highest failure rates. %
We observe that the relative log-error in the approximation is less than $10^{-2}$.
Fig.~\ref{fig:rates:capacity_utilization} shows that lines with higher utilization of their flow capacity, defined as $\sqrt{\Theta_l(x, y)}/I_l^\mathrm{lim} \times 100\%$, do not necessarily have higher failure rates, thus illustrating that \textit{the magnitude of line flow is not an effective surrogate for its failure rate}.
For example the line with the second-largest capacity utilization has a negligible failure rate, whereas lines with the largest failure rates have capacity utilization close to 60\%.
For accuracy, we must thus work directly with the failure rates via~\eqref{eq:failure_rate_first}.

\begin{figure}[!t]
    \centering
    \subfloat[Failure rate approximation error\label{fig:rates:approximation_quality}]{%
        \resizebox{0.48\linewidth}{!}{\includegraphics{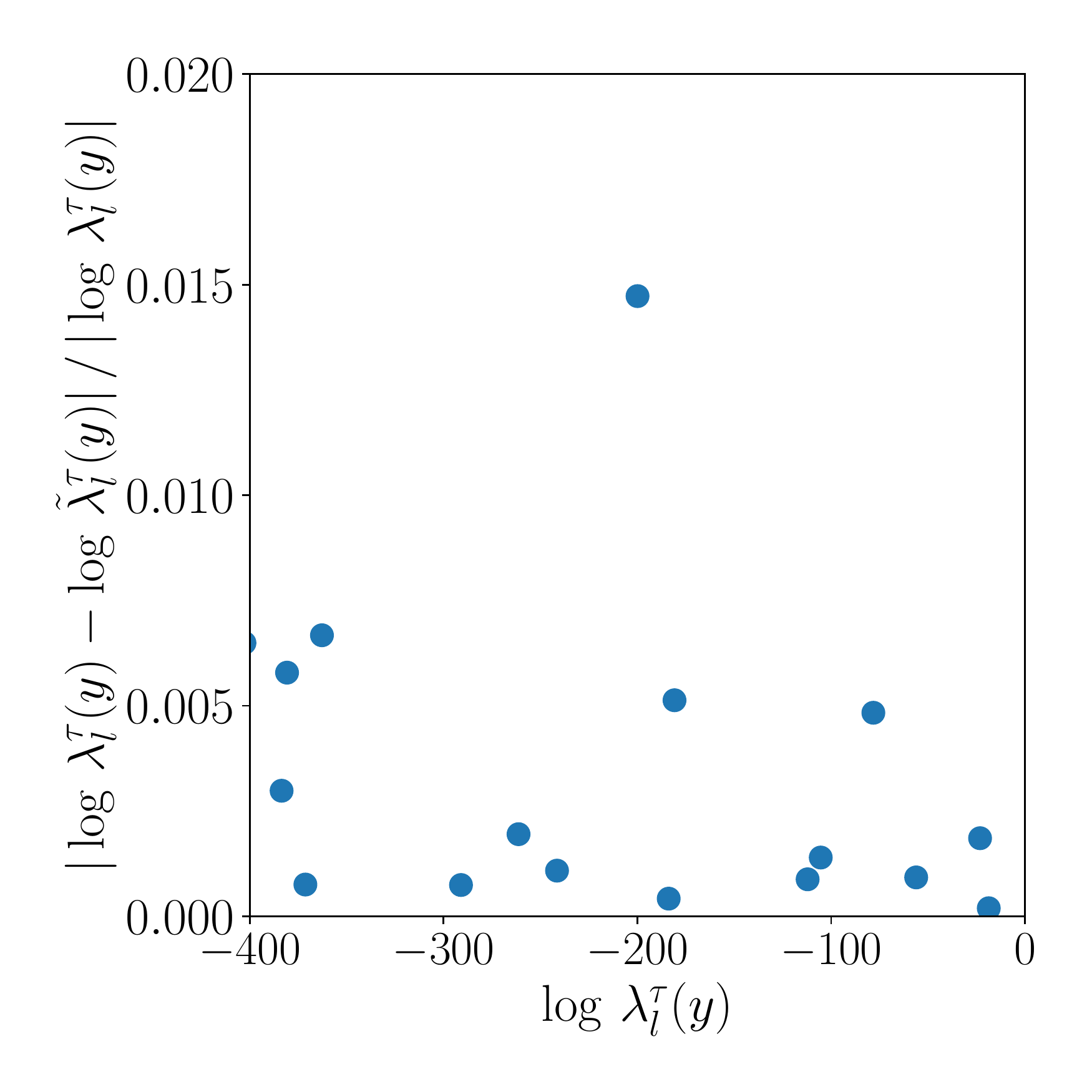}}
    }
    \hfil
    \subfloat[Capacity utilization vs. failure rate\label{fig:rates:capacity_utilization}]{%
        \resizebox{0.48\linewidth}{!}{\includegraphics{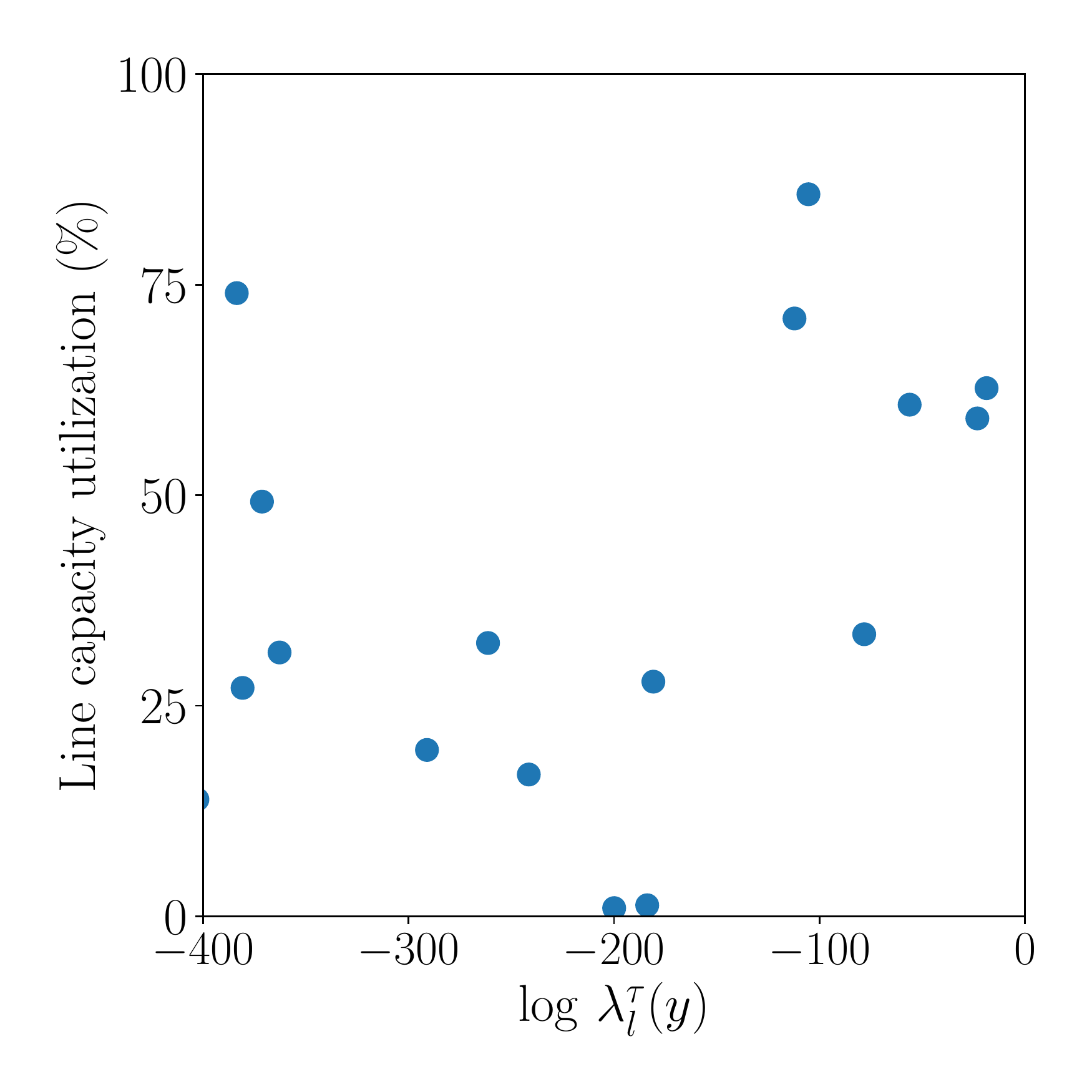}}
    }
    \caption{Fig.~\ref{fig:rates:approximation_quality} shows the relative log-error in approximating the true failure rate $\lambda^\tau_l (y)$ by $\tilde{\lambda}^\tau_l (y)$ in the rate-constrained model. %
        Fig.~\ref{fig:rates:capacity_utilization} shows line capacity utilization as a function of its failure rate.
        Each dot corresponds to a transmission line and $y$ is the classical (\emph{i.e.}, $N-0$) ACOPF dispatch.
    }
    \label{fig:rates}
\end{figure}

\subsection{Numerical and economic performance}\label{sec:numerical_simulations:computational_efficiency}
Table~\ref{table:numerical_performance_summary} summarizes the computational time and generation cost of the failure probability-constrained, the $N-0$, and the $N-1$ security-constrained\footnote{The set of contingencies is comprised of all lines, and post-contingency generation is allowed to vary by at most $0.1 p_{g,k}^\mathrm{max}$ at each generator $k \in \mathcal{G}$.} dispatch points for the 118-bus test system.
We observe that the FP-ACOPF solves within an order of magnitude of the time it takes to solve the $N-0$ ACOPF.
On the other hand, the $N-1$ model is slower by more than two orders of magnitude, although this may be because all contingencies are included in a single optimization problem.
In terms of generation cost, including the rate constraints results in a cost increase of less than 0.01\% compared to the $N-0$ model, even with extremely stringent rate limits, without requiring any load shedding.

It is worth pointing out the higher cost of the $N-1$ model comes with a guarantee of $N-1$ reliability.
Based on a contingency analysis of the FP-ACOPF dispatch points, we found that they remained feasible in $184$ (out of $186$) line contingencies, even though this was not explicitly imposed in their optimization.
Further analysis found that they could provide feasible power flows under all contingencies, but failed only to satisfy the thermal line limits \eqref{eq:line_flow_bounds} in the two infeasible contingencies.
In practice, operators may impose the failure rate constraints in each contingency of the $N-1$ ACOPF, thus ensuring both $N-1$ reliability and low failure probability.

\begin{table}[!t]
\caption{Summary of numerical performance for the 118-bus test case}
\label{table:numerical_performance_summary}
\centering
\begin{tabularx}{\columnwidth}{cccc}
    \toprule
    Formulation & $\displaystyle\max_{l \in \mathcal{L}'} \lambda_l^\tau(y)$ & Solve time (sec)  & Cost diff (\%) \\
    \midrule\midrule
    $N-0$                             & $7.8 \times 10^{-9}$  & 0.1   & -- \\
    $N-1$                             & $3.5 \times 10^{-9}$  & 73.0  & 0.04 \\
    $\lambda^\mathrm{lim} = 10^{-9}$  & $9.6 \times 10^{-11}$ & 2.7   & 0.00 \\
    $\lambda^\mathrm{lim} = 10^{-12}$ & $1.0 \times 10^{-12}$ & 2.8   & 0.00 \\
    $\lambda^\mathrm{lim} = 10^{-15}$ & $9.9 \times 10^{-16}$ & 2.3   & 0.01 \\
    \bottomrule
\end{tabularx}
\end{table}

Table~\ref{table:scalability_summary} summarizes the computational times of the FP-ACOPF for various PGLib test cases of increasing size.
Due to the distinct structures of these test networks, we impose failure rate limits that are a factor of 10 and 100 smaller than the maximum failure rate of their $N-0$ ACOPF dispatch point.
For some of the cases, the imposed rate limits may be so stringent that the corresponding FP-ACOPF formulation can become infeasible.
To recover a dispatch point in such cases, we first add load shedding (slack) variables $p_{s}, q_{s} \in \mathbb{R}^{n_b}$ to the power flow equations \eqref{eq:real_power_balance}--\eqref{eq:imag_power_balance} as follows:
\begin{gather*}
    p_{net,i} + \sum_{j \in \mathcal{N}} B_{ij} V_i V_j \sin(\theta_i - \theta_j) = p_{s,i}, \;\; i \in \mathcal{N}, \\
    q_{net,i} - \sum_{j \in \mathcal{N}} B_{ij} V_i V_j \cos(\theta_i - \theta_j) = q_{s,i}, \;\; i \in \mathcal{N},
\end{gather*}
and then minimize them in the objective function \eqref{eq:generation_cost} using a large quadratic penalty ($\phi = 10^6$) as follows:
\begin{equation*}
    \sum_{k \in \mathcal{G}} c_k(p_{g,k}) + \phi \sum_{i \in \mathcal{N}} \left(p_{s,i}^2 + q_{s,i}^2\right).
\end{equation*}
The column ``Load shed'' in Table~\ref{table:scalability_summary} reports the absolute values of these variables (as a fraction of the total system load).
The columns ``NLP time'' and ``Eval time'' report the total (wall-clock) time spent in Steps 2 and 3 of the algorithm described at the end of Section~\ref{sec:optimization_model}, respectively.

We make the following observations from Table~\ref{table:scalability_summary}.
First, when compared to the $N-0$ ACOPF, the time spent to solve the FP-ACOPF formulation increases roughly by a factor of 10 for both loose rate limits (across all cases) as well as tighter rate limits (up to the 500-bus case).
In contrast, for the tighter rate limit in the 1354-bus case, the time increases from roughly 3~sec to 560~sec.
This increase can be explained by the observation that for this particular case, Step~2 was invoked 3~times, each solving a nonlinear problem with an increasing number of failure rate constraints.
Second, the time spent in Step~3 to screen and evaluate violated failure rate constraints is an increasing function of the number of network lines.
Third, a nonzero (but small $\sim$2\%) amount of load must be shed in the 200-bus and 1354-bus cases to satisfy the imposed rate limits.
However, this phenomenon is also true in the $N-1$ ACOPF where it is common to allow constraint violations to ensure feasibility \cite{gocompetition}.
Finally, we note that the ``NLP time'' can be reduced using decomposition techniques that are used to solve the $N-1$ ACOPF \cite{chiang2014structured}, whereas the ``Eval time'' can be reduced using a larger number of parallel computing processes.
Indeed, although not shown in the table, increasing the number of processes from 32 to 64 decreases the evaluation time in the largest case from 557~sec to roughly 370~sec.

\begin{table*}[!htb]
    \caption{Summary of numerical performance across PGLib test cases\label{table:scalability_summary}}
    \centering%
    \begin{tabularx}{\linewidth}{cR*{6}{R}}
        \toprule
        Case & $N-0$ & \multicolumn{3}{c}{$\displaystyle\lambda^\mathrm{lim} = 10^{-1} \times \max_{l \in \mathcal{L}'} \lambda_l^\tau\,(N-0\ \mathrm{dispatch})$}  & \multicolumn{3}{c}{$\displaystyle\lambda^\mathrm{lim} = 10^{-2} \times \max_{l \in \mathcal{L}'} \lambda_l^\tau\,(N-0\ \mathrm{dispatch})$} \\
        \cmidrule(r){2-2}\cmidrule(lr){3-5}\cmidrule(l){6-8}
             & NLP time (sec) & NLP time (sec) & Eval time (sec) & Load shed (\%) & NLP time (sec) & Eval time (sec) & Load shed (\%) \\
        \midrule\midrule
         118 &      0.1 &      0.7 &      0.7 &     0.00 &      0.6 &      1.1 &     0.00 \\
         200 &      0.1 &      2.8 &      2.0 &     0.79 &      8.2 &      2.3 &     2.10 \\
         300 &      0.3 &      3.3 &      4.7 &     0.00 &      2.9 &      5.8 &     0.00 \\
         500 &      0.6 &      5.2 &     17.7 &     0.00 &     10.5 &     18.4 &     0.00 \\
        1354 &      3.2 &     19.9 &    351.0 &     0.11 &    554.2 &    557.0 &     0.19 \\
        \bottomrule
    \end{tabularx}%
\end{table*}

\subsection{Cascading failure simulations of different dispatch points}\label{sec:numerical_simulations:kmc_simulations}

To compare the cascade potentials of the failure probability-constrained, the $N-0$, and the $N-1$ security-constrained ACOPF, we perform $1000$ cascading failure simulations per dispatch point, using the Kinetic Monte Carlo (KMC) simulator developed in~\cite{roth2019kinetic}. %
In each simulation, the original network is randomly perturbed with two initial line contingencies\footnote{Initial contingencies are sampled as follows: \emph{(i)} choose the first contingency uniformly at random from all available lines, \emph{(ii)} sample a realization $d$ from a Zipf distribution with parameters $s = 3$ and $N = 10$, \emph{(iii)} group all lines that are topological distance $d$ from the first initial contingency, and \emph{(iv)} uniformly sample the second contingency from this group.}.

For a given initial network state and dispatch point, KMC simulates cascading sequences of line failures by using the failure rate expression in~\eqref{eq:failure_rate_first} to identify the next likely line failure.
Besides the initial dispatch point, 
KMC assumes that no operator interventions take place over the cascading time horizon, which we set to be one~hour based on the typical time scale of operator actions \cite{ren2008using,dobson_2012_estimating}.
Moreover, it assumes that after a failure occurs, the system `re-equilibriates' to a new static equilibrium (satisfying the power flow equations) before the next failure occurs.
That is, the system reaches a post-failure state (without triggering any intervening failures) in an amount of time that is small relative to the time until the next failure.
This assumption is implicit in the majority of existing quasi-steady-state models \cite{henneaux2018benchmarking} and numerical evidence in support of this assumption can be found in \cite{roth2019kinetic}.

Each post-failure state satisfies the power flow equations and hence is a local minimizer~\eqref{eq:xbar} of the energy function, appropriately modified (via the admittance matrix) to account for the degraded network topology.
Following each line failure, the slack bus is assumed to maintain power balance in the network, and buses that no longer connect to the slack bus (under the modified network topology) are disconnected, as well as buses whose voltage magnitudes fall below 0.9 p.u.
The cumulative active demand $p_{d}$ at disconnected buses serves as a measure of lost load.
This enables fair comparison of different dispatch points, on the basis of time and severity until full system collapse or until 
the end of the cascading time horizon (one hour). %

    Fig.~\ref{fig:kmc_failed_lines_load_shed_120} shows the distributions of the number of failed lines and the amount of load shed for each dispatch point, in the form of survival functions, for baseline values of $\tau = 10^{-4}$, $I^\mathrm{trip}_l = 1.10\, I^\mathrm{lim}_l$ and nominal MATPOWER $(p_d, q_d)$.
    Their averages and standard deviations are summarized in Table~\ref{table:kmc_stats}.

    Fig.~\ref{fig:kmc_log_failed_lines_cdf_120} shows the probability that the total number of line failures is at least as large as a given value, as a proportion of all $1000$ simulations.
    Here, we observe that all survival functions decrease as the number of line failures increases, and this decrease accelerates when the number of failures is large, indicating that the frequency of large blackouts decreases rapidly.
    Moreover, unlike the $N-1$ model, the FP-ACOPF dispatch can be tuned via the rate limits $\lambda^\text{lim}$, and for sufficiently stringent $\lambda^\text{lim} \in \{10^{-12}, 10^{-15}\}$, it can reduce the average number of line failures by more than 50\% when compared to the $N-0$ and $N-1$ models, as seen in Table~\ref{table:kmc_stats}.%

\begin{figure*}[!htb]
    \centering
    \subfloat[Distribution of number of failed lines \label{fig:kmc_log_failed_lines_cdf_120}]{%
        \resizebox{0.45\linewidth}{!}{\includegraphics{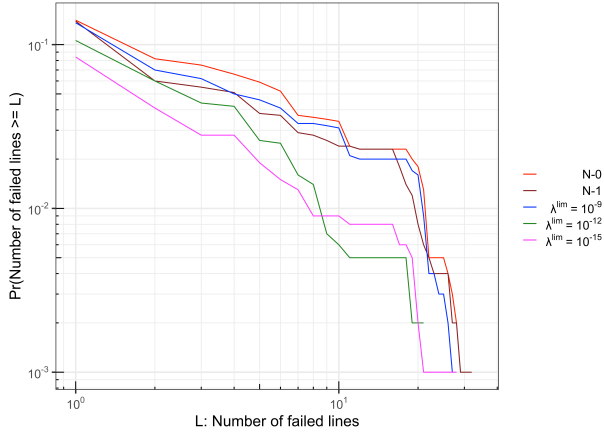}}
    }\hfil
    \subfloat[Distribution of load shed (nominal load $\approx 42$ p.u.) \label{fig:kmc_log_load_shed_cdf_120}]{%
        \resizebox{0.45\linewidth}{!}{\includegraphics{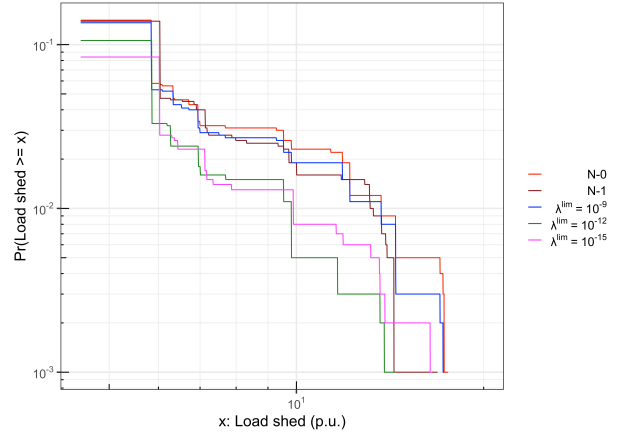}}
    }
    \caption{Distribution of number of failed lines and load shed using the KMC simulator, considering all simulations.}
    \label{fig:kmc_failed_lines_load_shed_120}
\end{figure*}

\begin{table}[!ht]
    \caption{Average number of failed lines and load shed (standard deviation in parentheses) over all KMC simulations}%
    \label{table:kmc_stats}
    \centering
    \begin{tabularx}{0.9\columnwidth}{ccc}
        \toprule
        Formulation & Number of failed lines & Load shed (p.u.) \\
        \midrule\midrule
        $N-0$                             & {$0.88\;\;(3.52)$} & {$1.06\;\;(2.85)$} \\ %
        $N-1$                             & {$0.71\;\;(3.17)$} & {$1.02\;\;(2.71)          $} \\ %
        $\lambda^\mathrm{lim} = 10^{-9}$  & {$0.75\;\;(3.21)$} & {$1.00\;\;(2.72)$} \\ %
        $\lambda^\mathrm{lim} = 10^{-12}$ & {$0.39\;\;(1.78)$} & {$0.71\;\;(2.15)$} \\ %
        $\lambda^\mathrm{lim} = 10^{-15}$ & {$0.33\;\;(1.92)$} & {$0.60\;\;(2.11)          $} \\ %
        \bottomrule
    \end{tabularx}
\end{table}

    Fig.~\ref{fig:kmc_log_load_shed_cdf_120} shows the distribution of load shed for each dispatch point as a probability of the load shed being at least as large as a given value, over all $1000$ simulations.
    We find that the $N-1$ security-constrained ACOPF performs as well as the FP-ACOPF in reducing the probability of very large load shedding, which is not surprising since the FP-ACOPF formulation explicitly constrains the probability of only the first line failure.
    However, unlike the former, the latter is also able to reduce the probability of small-to-intermediate demand losses, leading to a more than 30\% reduction in average load shed for $\lambda^\text{lim} \in \{10^{-12}, 10^{-15}\}$.

Fig.~\ref{fig:kmc_time_cdf_120} shows the average time until a certain number of lines have failed.
We find that, by delaying the system's first failure time, the FP-ACOPF is able to successfully delay--and hence, prevent--subsequent line failures which could lead to system collapse. For example, in the first 15 minutes after the first failure, FP-ACOPF with $\lambda^\mathrm{lim} = 10^{-15}$ shows about two times fewer lines failures compared to the $N-1$ dispatch.

\begin{figure}[!htbp]
    \centering
    \includegraphics[width=\linewidth]{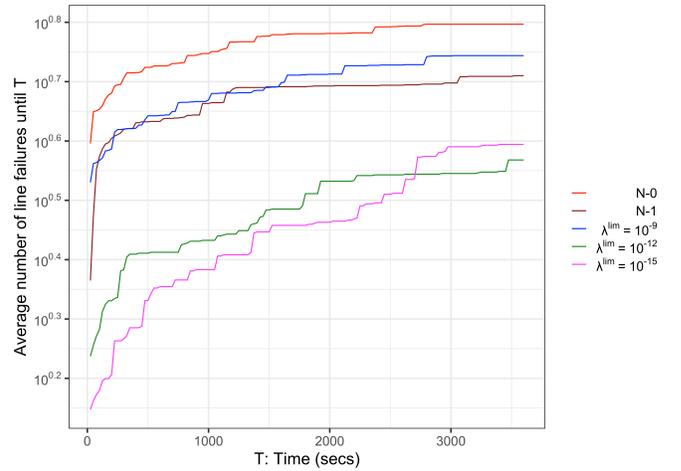}
    \caption{Average timeline of line failures (averaged across simulations where at least one failure occurred).}
    \label{fig:kmc_time_cdf_120}
\end{figure}

\subsection{Validation and benchmarking of cascade statistics}\label{sec:numerical_simulations:accfm_simulations}

The goal of this section is to validate the simulation statistics from the previous section.
We proceed to do this in three ways: 1) qualitative comparisons with historical cascades, 2) quantitative comparisons with an alternate AC-based cascade simulator that is based on an entirely different methodology, and finally, 3) validation of some of the key assumptions.

\subsubsection{Historical cascades}

The number of line failures and load shed in historical outage data are known to follow a heavy-tailed or Zipf distribution \cite{dobson_2012_estimating,henneaux2018benchmarking}.
This pattern can also be seen in Figs.~\ref{fig:kmc_log_failed_lines_cdf_120} and~\ref{fig:kmc_log_load_shed_cdf_120}; specifically, each survival function decreases roughly linearly in these log-log plots, which is characteristic of a power law distribution.
Although this linear behavior is more pronounced in Fig.~\ref{fig:kmc_log_failed_lines_cdf_120}, it is also evident for an intermediate range of values in Fig.~\ref{fig:kmc_log_load_shed_cdf_120}.
In addition, the distinctive `knee' in Fig.~\ref{fig:kmc_log_failed_lines_cdf_120} has been observed in other cascading failure models as well \cite{henneaux2018benchmarking}. %

Historical cascades also show a distinctive acceleration during their initial stages \cite{dobson_2012_estimating,noebels2020ac}.
This is somewhat evident in Fig.~\ref{fig:kmc_time_cdf_120} where the $N-1$ and FP-ACOPF dispatch points with $\lambda^\text{lim} \in \{10^{-12}, 10^{-15}\}$ exhibit a rapid succession of failures in the beginning but then appear to taper off.

We note that even though our simulations cannot be directly compared to historical data (since they correspond to different systems and initiating events), the above observations indicate that the KMC cascading failure model captures important patterns that are representative of historical outages.
A more thorough validation and comparison against real outage data can be found in \cite{roth2019kinetic}.

\subsubsection{Comparison with an alternate cascading failure simulator}

We now compare the cascade potentials of the various dispatch points using AC-CFM, an existing AC-based quasi-steady-state simulator developed in~\cite{noebels2020ac}.
We perform $1000$ simulations per dispatch point, where the original network (which we adjusted to be lossless) is randomly perturbed with the same two initial line contingencies as in Section~\ref{sec:numerical_simulations:kmc_simulations}.
In contrast to the KMC simulator, AC-CFM uses a set of deterministic rules to model various protection mechanisms.
In addition to line trips (which we adjusted per eq.~\eqref{eq:line_energy} with $I_{l}^{\mathrm{trip}}=1.10$) and under-voltage load shedding---mechanisms modeled by KMC---it also models under-frequency load (and over-frequency generator) shedding, and under- and over-excitation limiters.

\begin{figure*}[!htb]
    \centering
    \subfloat[Distribution of number of failed lines\label{fig:accfm_failed_lines_cdf_120}]{%
        \resizebox{0.45\linewidth}{!}{\includegraphics{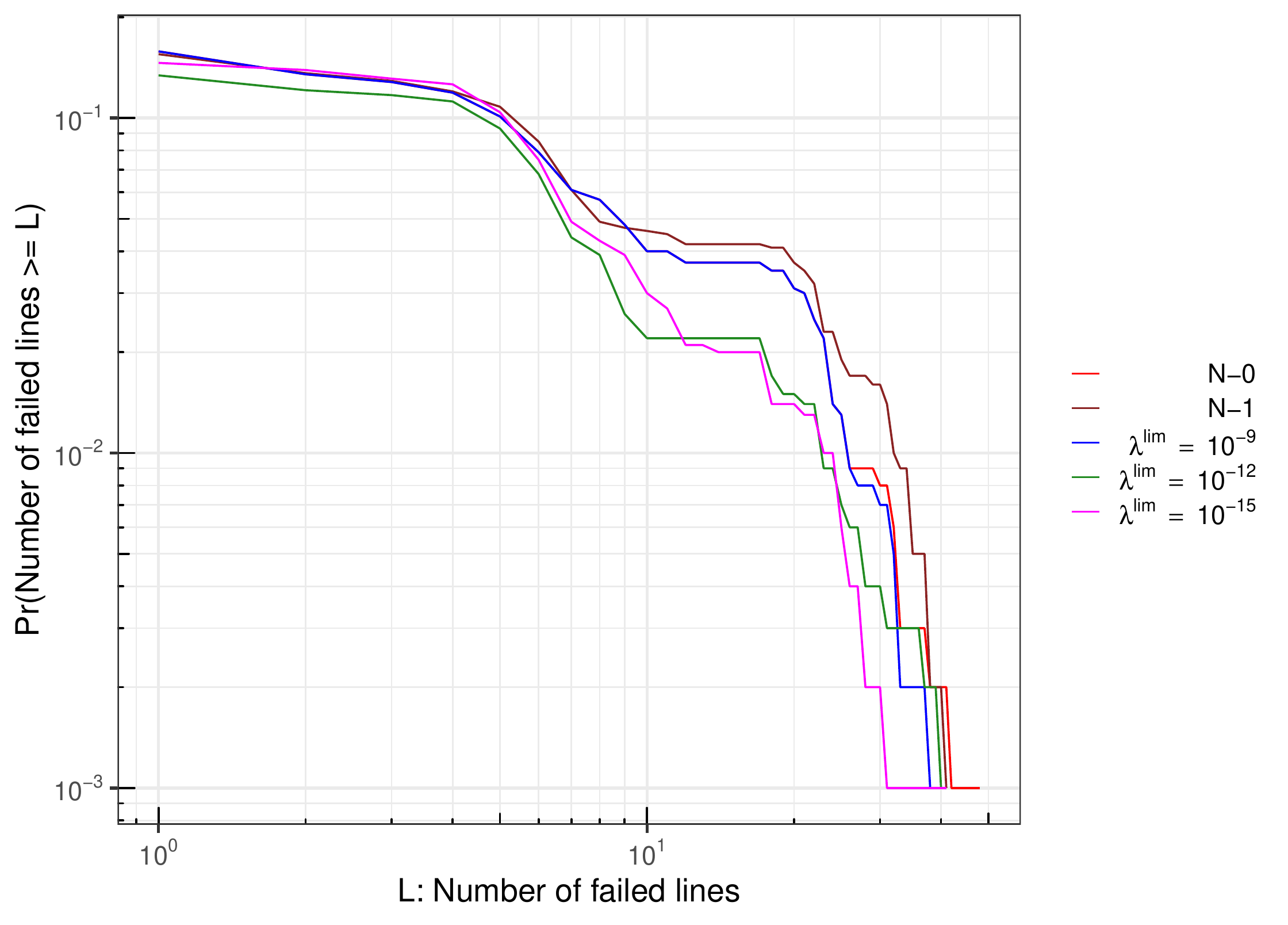}}
    }\hfil
    \subfloat[Distribution of load shed (nominal load $\approx 42$ p.u.) \label{fig:accfm_load_shed_cdf_120}]{%
        \resizebox{0.45\linewidth}{!}{\includegraphics{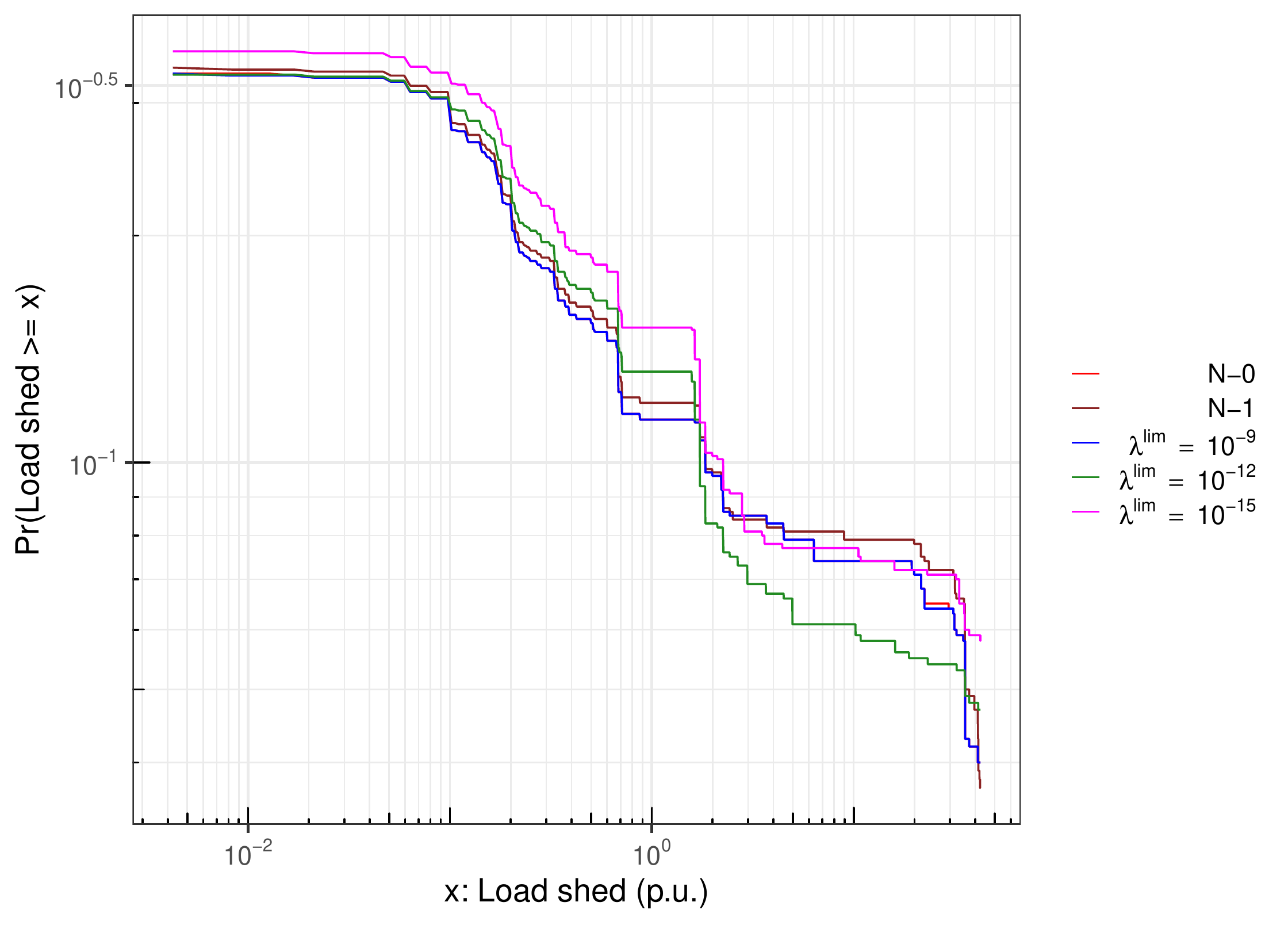}}
    }
    \caption{Distribution of number of failed lines and load shed using the AC-CFM simulator, considering all simulations.}
    \label{fig:accfm}
\end{figure*}

Figs.~\ref{fig:accfm_failed_lines_cdf_120} and~\ref{fig:accfm_load_shed_cdf_120} show the distribution of the number of failed lines and total load shed for each dispatch point, respectively, as now predicted by AC-CFM.
For small values of the number of failed lines in Fig.~\ref{fig:accfm_failed_lines_cdf_120}, the survival functions have a similar shape but begin at different levels.

Examining the intercept of Fig.~\ref{fig:accfm_failed_lines_cdf_120} shows that, relative to $N-0$,  the dispatch points corresponding to FP-ACOPF with $\lambda^{\mathrm{lim}} \in\{10^{-12}, 10^{-15}\}$ demonstrate a 15\% and 8\% reduction, respectively, in the number of observed cascades with at least one failure, compared to a 2\% reduction demonstrated by $N-1$.
In addition, Fig.~\ref{fig:accfm_failed_lines_cdf_120} indicates that the FP-ACOPF dispatch points have lower failure probabilities for intermediate-to-large numbers of failed lines.
Though no full system failures were observed, this pattern can also be found in Fig.~\ref{fig:accfm_load_shed_cdf_120} for load shed greater than $1$ p.u, indicating that the FP-ACOPF dispatch points can also reduce the frequency of intermediate-to-large demand losses at the expense of possibly inducing a higher proportion of smaller losses.

Finally, Table~\ref{table:accfm_stats} shows that across all simulations, the best FP-ACOPF dispatch point reduced the average number of line failures by more than 25\% (and had lower variability) and reduced the average amount of load lost by more than 10\%, compared to $N-0$ ACOPF.
In contrast, the $N-1$ dispatch point \emph{increased} the number of failures by 10\% (with larger variability) and \emph{increased} the amount of load lost by 8\%.

\begin{table}[!htb]
    \caption{Average number of failed lines and load shed (standard deviation in parentheses) over all AC-CFM simulations}%
    \label{table:accfm_stats}
    \centering
    \begin{tabularx}{0.9\columnwidth}{ccc}
        \toprule
        Formulation & Number of failed lines & Load shed (p.u.) \\
        \midrule\midrule
        $N-0$                             & {$1.48\;\;(5.08)$} & {$2.94\;\;\phantom{1}(9.94)$} \\ %
        $N-1$                             & {$1.63\;\;(5.62)$} & {$3.17\;\;(10.41)          $} \\ %
        $\lambda^\mathrm{lim} = 10^{-9}$  & {$1.46\;\;(4.92)$} & {$2.93\;\;\phantom{1}(9.93)$} \\ %
        $\lambda^\mathrm{lim} = 10^{-12}$ & {$1.08\;\;(3.96)$} & {$2.56\;\;\phantom{1}(9.46)$} \\ %
        $\lambda^\mathrm{lim} = 10^{-15}$ & {$1.15\;\;(3.80)$} & {$3.21\;\;(10.58)          $} \\ %
        \bottomrule
    \end{tabularx}
\end{table}

Additional experiments across a range of AC-CFM parameters, including which protection mechanisms are modeled along with their thresholds, point to the fact that it is difficult to directly compare absolute estimates of failure probability with the KMC simulator.
This is not surprising since they are based on different assumptions, and similar observations across a large range of simulators have been made in \cite{henneaux2018benchmarking}.
Nevertheless, we find it encouraging that the various dispatch points exhibit similar patterns (in particular in terms of their relative ordering with respect to cascade sizes) even when evaluated using a cascading methodology, AC-CFM, that is different from our KMC approach, and that FP-ACOPF can lead to lower failure probabilities under certain conditions.

    \subsubsection{Validation of assumptions via transient stability analysis}\label{sec:numerical_simulations:dynamics}
    The KMC and AC-CFM simulators are both based on quasi-steady-state methodologies.
    They implicitly assume not only that the system stabilizes to a new static equilibrium (satisfying the power flow equations), but also that this happens in a time frame that is small relative to the time until the next failure.
    However, these assumptions can fail because of system dynamics that are not modeled in these simulators.
    For example, generator desynchronization induced by an outage may cause the system to become unstable or to stabilize to another state than the one predicted by the simulators.
    To validate these assumptions, we use an existing tool \cite{PowerDynamics2021} to perform stability analysis and verify: \textit{(i)} whether the FP-ACOPF dispatch point is a stable operating point; \textit{(ii)} whether the settling time following a failure (if the system remains stable) is relatively smaller than the expected next failure time as predicted by equation \eqref{eq:failure_probability}; and \textit{(iii)} whether the new stable operating state is same as that predicted by equation \eqref{eq:xbar}.

    To that end, we consider the FP-ACOPF dispatch point ($\lambda^{\mathrm{lim}} = 10^{-9}$) and simulate system dynamics (using classical swing equations for generators and assuming constant active and reactive power outputs at load buses) for multiple fault scenarios, each corresponding to a line outage.
    We found that in roughly 89\% of the (134 total) simulations, the system oscillations were damped to stable values in 27.3~seconds on average (with a maximum and standard deviation of 34.6 and 3.1~seconds, respectively), whereas only 11\% of simulations were reported as being potentially unstable.    
    Figure~\ref{fig:branch96_omega} illustrates the typical response of representative system variables following a line outage where the system remained stable.
    
    We also observed that, in all of the stable scenarios, the system stabilized to the same equilibrium state as the one predicted by equation \eqref{eq:xbar} (appropriately modified to account for the failed line).
    On the other hand, among the unstable scenarios where equation \eqref{eq:xbar} also converged to a solution, only roughly half of them turned out to be dynamically unstable.
    These experiments indicate that the aforementioned quasi-steady-state assumptions do hold for the majority of line outages, although this is likely to be true only up to a certain point where the system has not degraded significantly.
    At the same time, they also highlight the challenges and need to model system dynamics more accurately in the context of cascading failures \cite{song2015dynamic,schafer2018dynamically}.
    
    \begin{figure}[!htbp]
        \centering
        \includegraphics[width=\linewidth]{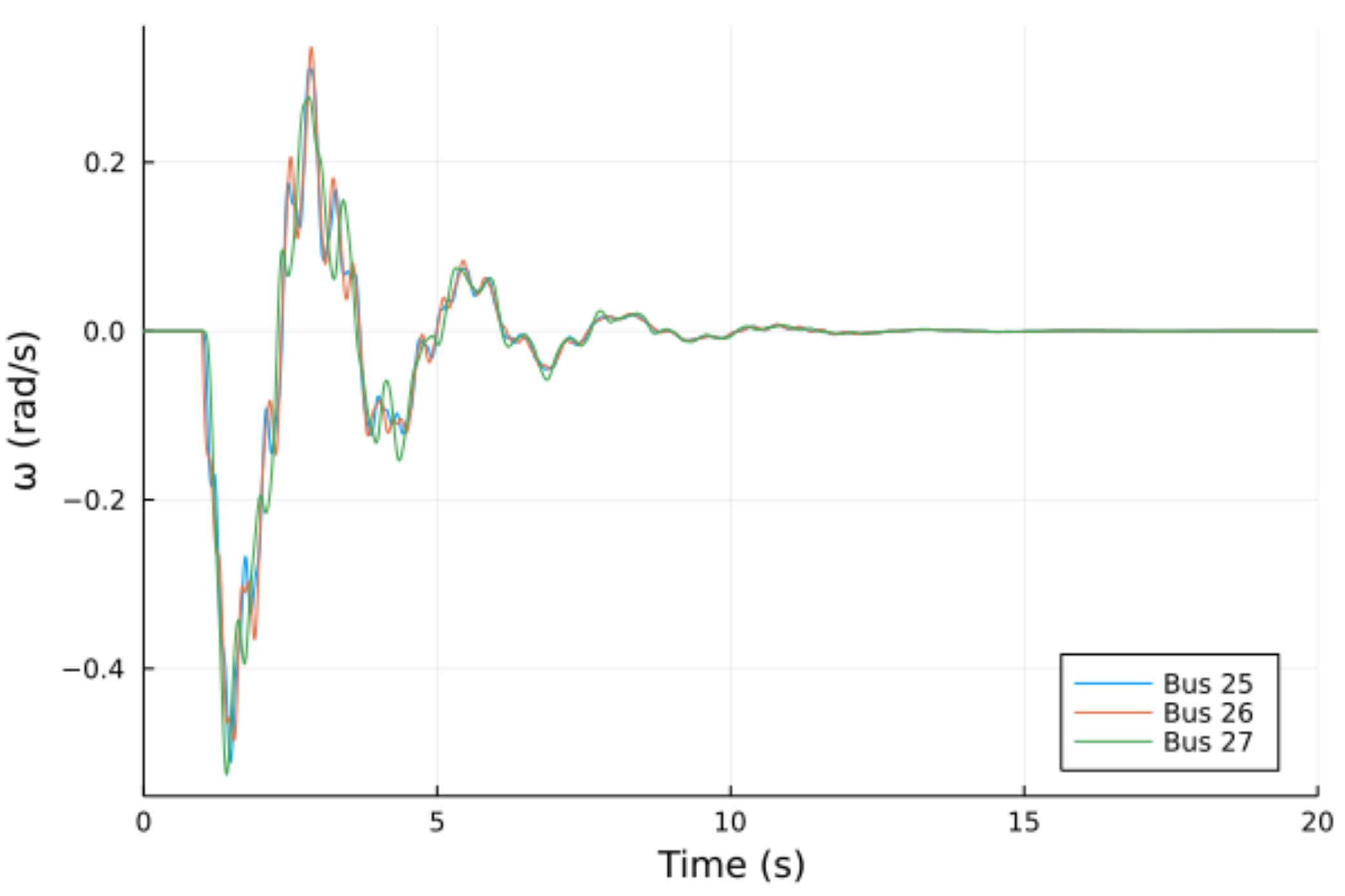}
        \caption{Representative generator angular velocities (relative to system frequency) following the outage of line $96$ %
            at time $t = 1$ second.\label{fig:branch96_omega}}
    \end{figure}

\subsection{Sensitivity analysis}\label{sec:numerical_simulations:sensitivity}

To study the sensitivity of our results to the choice of %
line failure threshold $I_l^\text{trip}$ and loads $(p_d, q_d)$, we repeated our analysis by varying them as follows:
$I_l^\mathrm{trip} \in \{1.05, 1.10, 1.12\} \times I_l^\mathrm{lim}$
and $(p_d, q_d)\in \{1.0, 1.05\}\times \mathrm{nominal}$ MATPOWER values%
, where the lower $I_l^\text{trip}$ and higher $(p_d, q_d)$ values were chosen so as to stress the system towards larger failures,
and the higher $I_l^\text{trip}$ value was chosen to elicit the opposite effect.
Fig.~\ref{fig:sensitivity} summarizes the KMC cascading failure simulations of the various dispatch points under these settings.

\begin{figure*}[!htb]
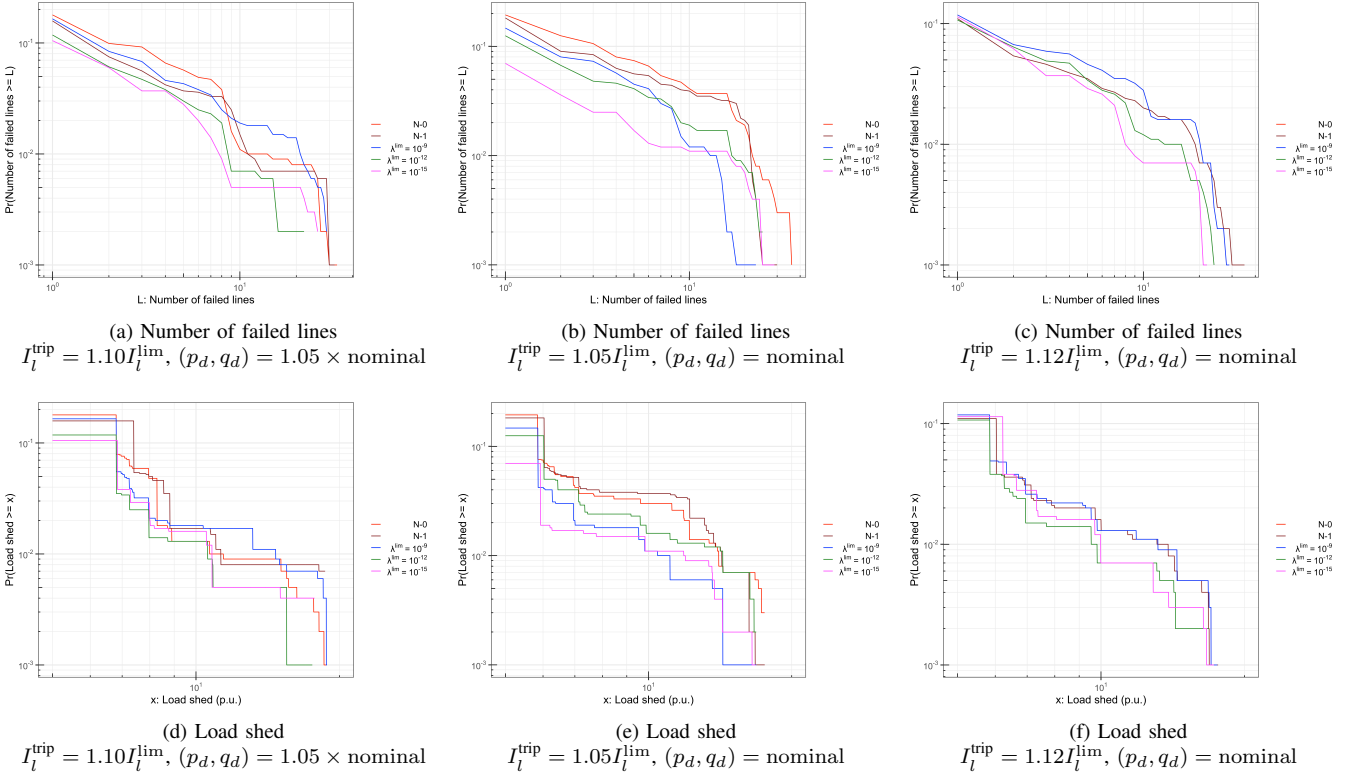

    \centering
    \captionsetup[subfigure]{justification=centering}
    \subfloat[][Number of failed lines \\ $I_l^\text{trip} = 1.10 I_l^\mathrm{lim}$,  $(p_d, q_d) = 1.05 \times \mathrm{nominal}$\label{fig:kmc_log_failed_lines_cdf_120_105}]{%
        \resizebox{0.32\linewidth}{!}{\includegraphics{figspdf/log_failed_lines_cdf_120_105}}
    }\hfil
    \subfloat[][Number of failed lines \\ $I_l^\text{trip} = 1.05 I_l^\mathrm{lim}$,  $(p_d, q_d) = \mathrm{nominal}$\label{fig:kmc_log_failed_lines_cdf_110}]{%
        \resizebox{0.32\linewidth}{!}{\includegraphics{figspdf/log_failed_lines_cdf_110}}
    }\hfil
    \subfloat[][Number of failed lines \\ $I_l^\text{trip} = 1.12 I_l^\mathrm{lim}$,  $(p_d, q_d) = \mathrm{nominal}$\label{fig:kmc_log_failed_lines_cdf_125}]{%
        \resizebox{0.32\linewidth}{!}{\includegraphics{figspdf/log_failed_lines_cdf_125}}
    }

    \subfloat[][Load shed \\ $I_l^\text{trip} = 1.10 I_l^\mathrm{lim}$,  $(p_d, q_d) = 1.05 \times \mathrm{nominal}$\label{fig:kmc_log_load_shed_cdf_120_105}]{%
        \resizebox{0.32\linewidth}{!}{\includegraphics{figspdf/log_load_shed_cdf_120_105}}
    }\hfil
    \subfloat[][Load shed \\ $I_l^\text{trip} = 1.05 I_l^\mathrm{lim}$,  $(p_d, q_d) = \mathrm{nominal}$\label{fig:kmc_log_load_shed_cdf_110}]{%
        \resizebox{0.32\linewidth}{!}{\includegraphics{figspdf/log_load_shed_cdf_110}}
    }\hfil
    \subfloat[][Load shed \\ $I_l^\text{trip} = 1.12 I_l^\mathrm{lim}$,  $(p_d, q_d) = \mathrm{nominal}$\label{fig:kmc_log_load_shed_cdf_125}]{%
        \resizebox{0.32\linewidth}{!}{\includegraphics{figspdf/log_load_shed_cdf_125}}
    }
    \caption{Distribution of number of failed lines and load shed for various values of line failure threshold $I_l^\text{trip}$ and system loads $(p_d, q_d)$ using KMC.}
    \label{fig:sensitivity}
\end{figure*}

First, we observe that the distribution of the number of failed lines (Figs.~\ref{fig:kmc_log_failed_lines_cdf_120_105}--\ref{fig:kmc_log_failed_lines_cdf_125}) for each dispatch point is an approximate power law in the range $L \in [1, 10]$, followed by a distinctive `knee' similar to Fig.~\ref{fig:kmc_log_failed_lines_cdf_120}, indicating a significant decrease in the probability of large line failures.
The distribution of the load shed (Figs.~\ref{fig:kmc_log_load_shed_cdf_120_105}--\ref{fig:kmc_log_load_shed_cdf_125}) is also an approximate power law, although for a smaller range of load shedding values.
Second, we observe that the FP-ACOPF generally outperforms the $N-0$ and $N-1$ ACOPF, both in terms of reduced probabilities of large line failures as well as load shedding.
Specifically, for higher load levels (Figs.~\ref{fig:kmc_log_failed_lines_cdf_120_105} and~\ref{fig:kmc_log_load_shed_cdf_120_105}), we find that the FP-ACOPF dispatches with $\lambda^\text{lim} \in \{10^{-12}, 10^{-15}\}$ perform better than the other dispatch points;
for lower line failure thresholds (Figs.~\ref{fig:kmc_log_failed_lines_cdf_110} and~\ref{fig:kmc_log_load_shed_cdf_110}),
the FP-ACOPF dispatches with $\lambda^\text{lim} \in \{10^{-9}, 10^{-15}\}$ perform significantly better, lowering failure probabilities by a factor of three on average, compared to both $N-0$ and $N-1$;
and for higher line failure thresholds (Figs.~\ref{fig:kmc_log_failed_lines_cdf_125} and~\ref{fig:kmc_log_load_shed_cdf_125}),
the FP-ACOPF dispatch points with $\lambda^\text{lim} \in \{10^{-12}, 10^{-15}\}$ have the lowest failure probabilities.
In this setting, we note that the survival functions of the $N-0$ dispatch coincide precisely with that of $\lambda^\text{lim} = 10^{-9}$. In summary, we find that the improvement in performance in regards to cascading risk of FP-ACOPF relative to $N-0$ and $N-1$ ACOPF appears to hold for a range or parameter choices beyond the ones derived from PGLib \cite{pglib} and MATPOWER \cite{matpower} that we used for the assessments in the preceding subsections.

\section{Conclusion and Future Work}\label{sec:conclusions}

This paper takes the first steps towards quantifying and proactively reducing failure risk--and implicitly, cascading risk--in  operational dispatch.
Our proposed FP-ACOPF generalizes the standard ACOPF formulation by constraining the probability of component failures due to automatic relay trips, via analytic functions of the system state and stochastic load fluctuations derived from our previous work \cite{roth2019kinetic}.
By using techniques from bilevel optimization and numerical linear algebra, we reformulated the probability constraints entirely algebraically allowing their solution using off-the-shelf nonlinear solvers.
We empirically showed that constraining failure probabilities--which is equivalent to increasing the system's expected first failure time--is a safe approximation and effective surrogate for constraining the probability of a cascading failure sequence starting from any single component outage.
Simulation outputs from two different cascading failure simulators provide evidence that FP-ACOPF can significantly outperform classical $N-0$ and $N-1$ security-constrained ACOPF models, in terms of reducing the probability of line outages and load shedding,
without incurring significantly larger computational or economic costs.

We envision future work along several directions.
From a methodological viewpoint, we need to extend the failure probability model and the corresponding optimization model to lossy networks. %
From a modeling viewpoint, more general multiple-component failure rate constraints, as well as extensions that combine classical $N-1$ models with failure rate constraints need to be investigated.
From a practical viewpoint, the model parameters (e.g., $\tau$) need to be calibrated and the model itself needs to be further validated and compared against other established cascading simulators on real network data. %
We believe these extensions can open up several other use cases for our method including long-term planning decisions, line capacity allocations, and contingency screening. %

\appendix
\section{}

\subsection{Closed-form expressions of low-rank factors}\label{appendix:closed_form_low_rank}

Consider the line overcurrent function~\eqref{eq:line_energy} for a line $l = (i, j)$.
First, consider the case where both $i, j \in \mathcal{N}'$ are non-generator buses, so that $V_i, V_j, \theta_i, \theta_j$ are all part of the state vector~$x$.
In this case, $r_l = 3$ and only the rows of $Q_l$ corresponding to these components have nonzero elements.
Suppose $1_{w} \in \mathbb{R}^d$ is the canonical unit vector with $1$ in the $w$-component and $0$ otherwise.
Then, it can be verified that
$
\nabla_{xx}^2 \Theta_l = Q_l K_l Q_l^\top,
$
where
$K_l = \mathrm{diag}\left(1, 1, -1\right)$
and the matrix $Q_l$ is as follows:
\begin{enumerate}
    \item\label{temp:q_0} $Q_l^\top 1_{w} = 0 \in \mathbb{R}^3$ for all $w \notin \{V_i, V_j, \theta_i, \theta_j\}$,
    \item $Q_l^\top \begin{bmatrix}
    1_{V_i} & 1_{V_j} & 1_{\theta_i} & 1_{\theta_j}
    \end{bmatrix}$ is given by:
    \begingroup
    \small
    \begin{gather*}
    \begin{bmatrix}
    1 & -c_{ij} & V_j s_{ij}       & -V_j s_{ij} \\
    0 &  s_{ij} & V_i + V_j c_{ij} & -V_i - V_j c_{ij} \\
    0 &  0      & \sqrt{V_i^2 + V_j^2 + V_i V_j c_{ij}} & -\sqrt{V_i^2 + V_j^2 + V_i V_j c_{ij}}
    \end{bmatrix} \\
    c_{ij} \coloneqq \cos(\theta_i - \theta_j), \;\;
    s_{ij} \coloneqq \sin(\theta_i - \theta_j).
    \end{gather*}
    \endgroup
\end{enumerate}

If $j \in \mathcal{N} \setminus (\mathcal{N}' \cup \{\sigma\})$ is a non-slack generator bus, then $V_i, \theta_i, \theta_j$ are part of $x$, but not $V_j$.
In this case, $r_l = 2$,
$K_l = \mathrm{diag}\left(1, V_i V_j c_{ij} - V_j^2 s_{ij}^2\right)$
and, along with condition 1) described above, the matrix $Q_l$ is characterized by:
\begingroup
\small
\begin{equation*}
    Q_l^\top \begin{bmatrix}
    1_{V_i} & 1_{\theta_i} & 1_{\theta_j}
    \end{bmatrix}
    =
    \begin{bmatrix}
    1 & V_j s_{ij} & -V_j s_{ij} \\
    0 & 1          & -1
    \end{bmatrix}
    \end{equation*}
\endgroup
If $j = \sigma$ is the slack bus, then $r_l = 2$, and the above expressions for $K_l$ and $Q_l$ continue to be applicable,
except $\theta_j$ is not part of $x$ and has no corresponding column in $Q_l$.

\subsection{Proof of Proposition~\ref*{prop:bilevel_reformulation}}\label{appendix:proof_bilevel}
We drop the subscript $l$ from all quantities to simplify the notation.
First, observe that if $x^\star$ satisfies equation \eqref{eq:bilevel_feasibility}, \emph{i.e.}, $\Theta (x^\star, y) = \Theta^{\max}$, then it must also satisfy $\nabla_x \Theta (x^\star, y) \neq 0$.
Indeed, by definition of $\Theta$ in \eqref{eq:line_energy}, it follows that $\nabla_x \Theta (x^\star, y) = 0$ implies $V_i^\star = V_j^\star = 0$.
This, in turn, implies $\Theta (x^\star, y) = 0$, which is a contradiction $\Theta^{\max} > 0$.
Hence, $\nabla_x \Theta (x^\star, y) \neq 0$, and the \textit{linearly independent constraint qualification} is satisfied for problem \eqref{eq:xstar} at $x^\star$.
Therefore, its first-order optimality conditions, which are \eqref{eq:bilevel_kkt_stationarity} using the Taylor approximation \eqref{eq:energy_function_quadratic} and equation \eqref{eq:bilevel_feasibility}, must be necessarily satisfied at $x^\star$.

Now, consider the Hessian of the Lagrangian of problem \eqref{eq:xstar} at the primal-dual pair $(x^\star, \mu^\star)$:
$
M \coloneqq \nabla_{xx}^2 \mathcal{H}(x, y) - \mu^\star \nabla_{xx}^2 \theta (x^\star, y),
$
where again we have used the Taylor approximation \eqref{eq:energy_function_quadratic}.
Along with the first-order conditions, the second-order sufficient conditions ensure that, if $M \succ 0$, then $x^\star$ is a local solution of problem \eqref{eq:xstar}.
In the following paragraph, we show that $M \succ 0$ is equivalent to condition \eqref{eq:bilevel_sosc}.

Denote $\nabla^2 \mathcal{H} \coloneqq \nabla_{xx}^2 \mathcal{H}(x, y)$, $Q^\star \coloneqq Q(x^\star, y)$ and $K^\star \coloneqq K(x^\star, y)$.
Using factorization \eqref{eq:low_rank_factorization} and by using the hypothesis $\nabla^2 \mathcal{H} \succ 0$ in the statement of the Proposition, we have
\begin{align*}
& M = \nabla^2 \mathcal{H} - \mu^\star Q^\star K^\star (Q^\star)^\top \succ 0 \\
\iff & I - \mu^\star [\nabla^2 \mathcal{H}]^{-1/2} Q^\star K^\star (Q^\star)^\top [\nabla^2 \mathcal{H}]^{-1/2} \succ 0 \\
\iff & \mu^\star\, \rho \left( [\nabla^2 \mathcal{H}]^{-1/2} Q^\star K^\star (Q^\star)^\top [\nabla^2 \mathcal{H}]^{-1/2} \right) < 1 \\
\iff & \mu^\star\, \rho \left( K^\star (Q^\star)^\top [\nabla^2 \mathcal{H}]^{-1} Q^\star \right) < 1.
\end{align*}
where the last equivalence from the fact that the non-zero eigenvalues of $AB$ coincide with those of $BA$ for arbitrary matrices $A$ and $B$ of appropriate size.

\subsection{Proof of Proposition~\ref*{prop:determinants}}\label{appendix:proof_determinant}
We again drop the subscript $l$ from all quantities to simplify notation.
Using equations \eqref{eq:energy_function_quadratic} and \eqref{eq:bilevel_kkt_stationarity}, we obtain:
\begin{align*}
\mathcal{H}(x^\star, y) - \mathcal{H}(x, y) &=  \frac12 (x^\star - x)^\top \nabla^2_{xx} \mathcal{H}(x, y) (x^\star - x) \\
&= \frac12 \mu^\star (x^\star - x)^\top \nabla_x \Theta (x^\star, y) = \frac12 \mu^\star \alpha,
\end{align*}
which shows the validity of \eqref{eq:prefactor_first_reformulated} and \eqref{eq:energy_factor_reformulated}.

Denote $\nabla^2 \mathcal{H} \coloneqq \nabla_{xx}^2 \mathcal{H}(x, y)$, $Q^\star \coloneqq Q(x^\star, y)$ and $K^\star \coloneqq K(x^\star, y)$.
To show validity of \eqref{eq:prefactor_zeroth_reformulated}, we will show that equation \eqref{eq:B-star} simplifies to $C^\star(y) = \mu^\star \det(\nabla^2 \mathcal{H}) \left( \alpha \det (W) + \beta\right)$.
The factorization \eqref{eq:low_rank_factorization} and Taylor approximation \eqref{eq:energy_function_quadratic} imply that~\eqref{eq:L-star} is equivalent to $X(y) = \nabla^2 \mathcal{H} - \mu^\star Q^\star K^\star (Q^\star)^\top$.
Furthermore, the proof of Proposition~\ref{prop:bilevel_reformulation} shows that $X(y) \succ 0$ and hence, that it is invertible.
Therefore, we have
\[
\mathop{\mathrm{adj}} (X(y)) = [X(y)]^{-1} \det(X(y)).
\]
Since $\nabla^2 \mathcal{H} \succ 0$, we can use the Woodbury identity to obtain:
\begin{align*}
[X(y)]^{-1} %
&= \left[\nabla^2 \mathcal{H}\right]^{-1} + \left[\nabla^2 \mathcal{H}\right]^{-1} \mu^\star Q^\star W K^\star (Q^\star)^\top \left[\nabla^2 \mathcal{H}\right]^{-1}
\end{align*}
Similarly, application of~\cite[Theorem~18.1.1]{Harville2008} gives:
\begin{align*}
\det(X(y)) %
&= \det (\nabla^2 \mathcal{H}) \det(W)
\end{align*}
Equation \eqref{eq:energy_function_quadratic} implies
$
\nabla \mathcal{H}^{\star}(y) = \nabla^2 \mathcal{H} \cdot (x^\star - x),
$
and along with the above equations, this simplifies \eqref{eq:B-star} as follows:
\begin{align*}
C^\star(y) &=  \nabla^\top \mathcal{H}^{\star}(y) \mathop{\mathrm{adj}} (X(y))  \nabla \mathcal{H}^{\star}(y) \\
&= \det (\nabla^2 \mathcal{H}) \det(W) \cdot \Big \{ (x^\star - x)^\top \, \nabla^2 \mathcal{H} \, (x^\star - x)  \\
&\hspace{4em} + \mu^\star (x^\star - x)^\top Q^\star W K^\star (Q^\star)^\top (x^\star - x)  \Big\} \\
&= \mu^\star \det(\nabla^2 \mathcal{H}) \left( \alpha \det (W) + \beta\right)
\end{align*}
where the last equality follows from equation \eqref{eq:bilevel_kkt_stationarity} and the definition of $\alpha$ and $\beta$.

\section*{Acknowledgment}
The authors would like to thank Prof. Ian Dobson from Iowa State University for help at various stages of this work, as well as Matthias Noebels from the University of Manchester for help with AC-CFM.  We thank the anonymous referees whose comments have significantly improved the paper. This material was based upon work supported by the U.S. Department of Energy, Office of Science, Office of Advanced Scientific Computing Research (ASCR) under Contract DE-AC02-06CH11347 and by NSF through award CNS-1545046.

\ifCLASSOPTIONcaptionsoff
  \newpage
\fi

\bibliographystyle{IEEEtran}
\bibliography{references_intro,references}

\begin{IEEEbiographynophoto}{Anirudh Subramanyam}
    is a postdoctoral researcher in the Mathematics and Computer Science Division at Argonne National Laboratory.
    He obtained his bachelor’s degree from the Indian Institute of Technology, Bombay and his Ph.D. from Carnegie Mellon University, both in chemical engineering. His research interests are in computational methods for nonlinear and discrete optimization under uncertainty with applications in energy, transportation and process systems.
\end{IEEEbiographynophoto}
\begin{IEEEbiographynophoto}{Jacob Roth}
  is a Ph.D. student in the ISyE department at the University of Minnesota. Previously, he received an M.S. from the University of Chicago and was a pre-doctoral appointee at Argonne National Laboratory.
\end{IEEEbiographynophoto}
\begin{IEEEbiographynophoto}{Albert Lam}
is a pre-doctoral appointee at Argonne National Laboratory. He received his B.S. from the University of New South Wales, and his M.S. from the University of Chicago.
\end{IEEEbiographynophoto}
\begin{IEEEbiographynophoto}{Mihai Anitescu}
    is a senior computational mathematician in the Mathematics
    and Computer Science Division at Argonne National Laboratory and a professor in the Department of Statistics at the University of Chicago. He obtained his engineer diploma (electrical engineering) from the Polytechnic University of Bucharest in 1992 and his Ph.D. in applied mathematical and computational sciences from the University of Iowa in 1997. He specializes in the areas of numerical optimization, computational science, numerical analysis and uncertainty quantification in which he has published more than 100 papers in scholarly journals and book chapters.
    He has been recognized for his work in applied mathematics by his selection as a SIAM Fellow in 2019.
\end{IEEEbiographynophoto}
\vfill
\noindent\fbox{\parbox{\columnwidth}{\footnotesize
        The submitted manuscript has been created by UChicago Argonne, LLC, Operator of Argonne National Laboratory (``Argonne''). Argonne, a U.S. Department of Energy Office of Science laboratory, is operated under Contract No. DE-AC02-06CH11357. The U.S. Government retains for itself, and others acting on its behalf, a paid-up nonexclusive, irrevocable worldwide license in said article to reproduce, prepare derivative works, distribute copies to the public, and perform publicly and display publicly, by or on behalf of the Government. The Department of Energy will provide public access to these results of federally sponsored research in accordance with the DOE Public Access Plan (http://energy.gov/downloads/doe-public-access-plan).}
}

\enlargethispage{-2mm}

\end{document}